\documentclass[english,reqno]{amsart}
\usepackage{kotex}
\usepackage{setspace}
\usepackage{amsmath,amssymb,bm}
\usepackage{xcolor}
\usepackage[left=3cm,right=3cm,top=3cm,bottom=3cm,headsep=1.5cm,footskip=1.5cm]{geometry}
\usepackage{amsthm}
\usepackage{mathrsfs}
\usepackage[noadjust]{cite}
\usepackage{enumitem}
	\setlist[itemize]{leftmargin=*}
	\setlist[enumerate]{leftmargin=*}
\usepackage{tikz}
\usetikzlibrary[patterns]
\usepackage{bbm}
\usepackage{stmaryrd}
\usepackage{amsaddr}
\usepackage[bookmarks=false,breaklinks=false,pdfborder={0 0 1},backref=false,colorlinks=false]{hyperref}

\numberwithin{equation}{section}
\numberwithin{figure}{section}
\theoremstyle{plain}
\newtheorem{thm}{Theorem}[section]
\newtheorem{lem}[thm]{Lemma}
\newtheorem{claim}[thm]{Claim}
\newtheorem{cor}[thm]{Corollary}
\theoremstyle{definition}
\newtheorem{defn}[thm]{Definition}
\theoremstyle{remark}
\newtheorem*{notation}{Notation}
\newtheorem{remark}[thm]{Remark}
\newtheorem*{acknowledgement}{Acknowledgement}
\begin{document}
\title[FEP in Higher Dimensions: Recurrent Structure \& Transient Dynamics]{Facilitated Exclusion Process in Higher Dimensions: Recurrent Structure and Transient Dynamics}
\author{Seonwoo Kim, Sanha Lee and Insuk Seo}
\address{S. Kim. Department of Mathematics, Yonsei University, South Korea.}
\email{seonwookim@yonsei.ac.kr}
\address{S. Lee. Department of Mathematical Sciences, Seoul National University, South Korea.}
\email{sanha7139@snu.ac.kr}
\address{I. Seo. Department of Mathematical Sciences and Research Institute of Mathematics,\\ Seoul National University, South Korea.}
\email{insuk.seo@snu.ac.kr}

\begin{abstract}
In this article, we study the facilitated exclusion process (FEP) on the $d$-dimensional discrete torus of side length $N$, with $d\ge2$. For Bernoulli initial data with a fixed density $\rho\in(0,1)$, we first prove that, with high probability, the particle-number sector selected by the initial configuration has a unique active recurrent class if $\rho>1-2^{-d}$ and multiple recurrent classes if $\rho<1-2^{-d}$. We then show that, for sufficiently small $\rho$, the process reaches an absorbing state (a singleton recurrent class) within a poly-logarithmic time in $N$ with high probability. In contrast, for $\rho\in(1/2,3/4)$ and even $N$, we establish a polynomial lower bound on the time required to reach the recurrent set. These results provide strong evidence for a double phase transition, analogous to absorbing-state phase transitions in the contact process, activated random walks, and stochastic sandpiles. To our knowledge, these are the first rigorous estimates for absorption and transient times for the FEP in dimensions two and higher.
\end{abstract}
\maketitle

\tableofcontents{}

\section{\label{sec1}Introduction}
\subsection{The Facilitated Exclusion Process and Main Questions}
The facilitated exclusion process (FEP) is a conservative particle system in which exclusion is supplemented by a local kinetic constraint. On the discrete torus $\mathbb{T}^{d}_{N}=(\mathbb{Z}/N\mathbb{Z})^{d}$ for $d\ge1$, each site is either vacant or occupied by one particle. A particle at $\bm{x}$ jumps to $\bm{x}+\bm{v}$, where $\bm{v}\in\{\pm\bm{e}_{1},\ldots,\pm\bm{e}_{d}\}$, at rate one provided that $\bm{x}+\bm{v}$ is vacant and $\bm{x}-\bm{v}$ is occupied. The particle at $\bm{x}-\bm{v}$ facilitates the jump; see Figure \ref{fig:jump-rule}. This rule preserves the number of particles but can prevent motion even in configurations containing both particles and vacancies. Consequently, at a fixed particle number there may be \emph{several recurrent communicating classes}. These may be singleton classes consisting of absorbing configurations or active classes that support persistent motion.

\begin{figure}[htbp]
\centering
\begin{tikzpicture}[
    x=0.90cm,
    y=0.90cm,
    >=stealth,
    occupied/.style={
        circle, draw, fill=black!65,
        inner sep=0pt, minimum size=7pt
    },
    vacant/.style={
        circle, draw, fill=white,
        inner sep=0pt, minimum size=7pt
    },
    jump/.style={->, semithick},
    description/.style={anchor=west, font=\small}
]
\foreach \yy in {0,-1.1,-2.2} {
    \draw[gray!50] (0,\yy) -- (2,\yy);
    \draw[jump]
        (1.15,{\yy+0.30}) -- (1.90,{\yy+0.30});
}
\node[occupied] at (0,0) {};
\node[occupied] at (1,0) {};
\node[vacant]   at (2,0) {};
\draw[jump]
    (2.6,0) -- node[above,font=\small]{rate $1$} (4.1,0);

\draw[gray!50] (4.7,0) -- (6.7,0);
\node[occupied] at (4.7,0) {};
\node[vacant]   at (5.7,0) {};
\node[occupied] at (6.7,0) {};

\node[description] at (7.3,0)
    {Allowed jump};
\node[vacant]   at (0,-1.1) {};
\node[occupied] at (1,-1.1) {};
\node[vacant]   at (2,-1.1) {};

\node[description] at (2.6,-1.1)
    {Blocked: no facilitating particle};
\node[occupied] at (0,-2.2) {};
\node[occupied] at (1,-2.2) {};
\node[occupied] at (2,-2.2) {};

\node[description] at (2.6,-2.2)
    {Blocked: occupied target};
\foreach \yy in {-1.1,-2.2} {
    \draw[semithick]
        (1.46,{\yy+0.17}) -- (1.62,{\yy+0.43});
}
\node[below,font=\small] at (0,-2.5)
    {$\bm{x}-\bm{v}$};
\node[below,font=\small] at (1,-2.5)
    {$\bm{x}$};
\node[below,font=\small] at (2,-2.5)
    {$\bm{x}+\bm{v}$};

\end{tikzpicture}
\caption{
The directional facilitation rule along three consecutive sites.
Filled and empty circles denote occupied and vacant sites, respectively.
Short arrows indicate proposed particle jumps from $\bm{x}$ to
$\bm{x}+\bm{v}$; a slash marks a blocked jump.
The long arrow shows the configuration change resulting from the
allowed jump, which occurs at rate $1$.
Only one direction is shown; the same rule applies to every
$\bm{v}\in\{\pm\bm{e}_1,\ldots,\pm\bm{e}_d\}$.
}
\label{fig:jump-rule}
\end{figure}
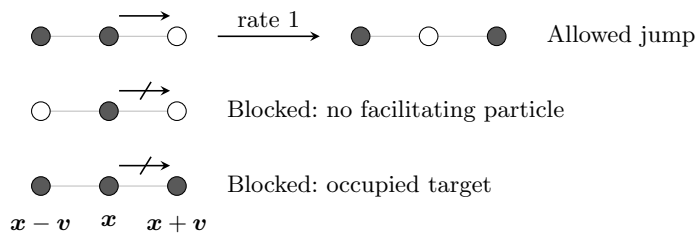

Two fundamental questions regarding FEP arise from this degeneracy: \emph{how does the recurrent decomposition depend on the particle density, and how long does the process take to reach its recurrent set from a typical initial configuration?} The first question concerns the geometry of configurations and the existence of legal paths between them, while the second concerns the stochastic evolution and the time needed to leave the transient states. These questions have been studied extensively in one dimension. In this paper, we investigate both questions in dimensions $d\geq2$.

\subsection{Known Results and Conjectures}
In one dimension, the recurrent structure is determined by the critical density $1/2$ (cf.\ \cite{BESS20}). Consider the FEP on $\mathbb{T}_{N}$ with $K$ particles, where $1\le K<N$.
\begin{itemize}
\item \emph{Supercritical regime.} If $N/2<K<N$, there is a unique active recurrent class, consisting of configurations with no neighboring vacancies.
\item \emph{Subcritical and critical regimes.} If $K\leq N/2$, every recurrent state is absorbing and has no neighboring particles.
\end{itemize}
The main questions in these two regimes are different. In the supercritical regime, a central objective is to describe the macroscopic evolution after the system reaches its unique active recurrent class. The one-dimensional supercritical hydrodynamic limit is governed by a fast diffusion equation, as proved in \cite{BESS20}. For more general initial profiles, the coexistence of frozen and active regions leads to a Stefan problem \cite{BES21}. A boundary-driven version has also been studied in \cite{DCES26}.

In the subcritical regime, the main question is how long the system takes to reach one of its absorbing states. This question is related to absorbing-state phase transitions, as studied in stochastic sandpiles \cite{ST17,VDMZ00}, activated random walks \cite{RS12,Rol20}, and the contact process \cite{Har74,BG90}. To describe the corresponding possible scenario for the one-dimensional FEP, suppose that the initial configuration has a Bernoulli product distribution with a fixed density $\rho\in(0,1/2)$. Here and below, positive constants may depend on $\rho$ and $d$, but not on $N$. A possible finite-volume signature of an absorbing-state phase transition within this interval would be the existence of a density $\rho_{\star}\in(0,1/2)$ such that the absorption time is at most poly-logarithmic in $N$, of order $(\log N)^{a}$ for some $a>0$, when $\rho<\rho_{\star}$, and at least exponential in $N$, of order $e^{bN}$ for some $b>0$, when $\rho>\rho_{\star}$, with high probability in both cases.

Recently, two of the present authors, together with Blondel and Erignoux \cite{BEKL26}, proved that such an absorbing-state phase transition \emph{does not occur} in the one-dimensional FEP. For Bernoulli initial data at every fixed density $\rho\in(0,1/2)$, the absorption time has order $(\log N)^{3}$ with high probability, i.e., with probability tending to $1$ as $N \to \infty$.

In higher dimensions, a dynamical transition is nevertheless expected. Numerical evidence for an absorbing-state phase transition in a related conservative lattice gas is discussed in \cite{ERSS24}. In that model, any occupied nearest neighbor can facilitate a jump, whereas the present model requires a particle at $\bm{x}-\bm{v}$ to facilitate the jump $\bm{x}\to\bm{x}+\bm{v}$. For the directional rule considered here, the role of absorption also requires some care: even at low particle densities, distinct active recurrent classes may coexist. Figure \ref{fig:active-rows} gives a simple example.

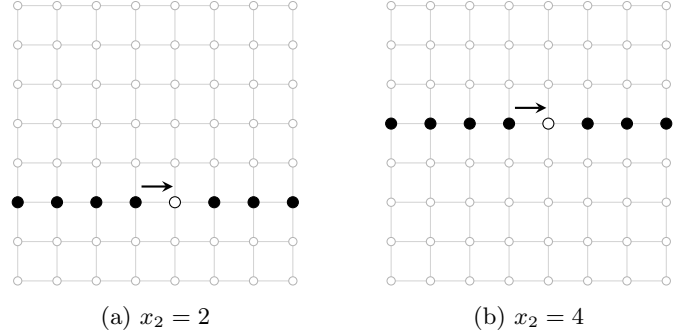
\begin{figure}[htbp]
\centering
\begin{tikzpicture}[scale=0.52, >=stealth]

\foreach \shift/\row/\name in {0/2/{(a) $x_2=2$}, 9.5/4/{(b) $x_2=4$}} 
{
  \begin{scope}[xshift=\shift cm]
  \draw[gray!35, line width=0.3pt, step=1]
        (0,0) grid (7,7);
    \foreach \x in {0,...,7} {
      \foreach \y in {0,...,7} {
        \filldraw[fill=white, draw=gray!60, line width=0.3pt] (\x,\y) circle (0.105);
      }
    }
    \foreach \x in {0,1,2,3,5,6,7} {\filldraw[fill=black, draw=black, line width=0.4pt] (\x,\row) circle (0.14); }
    \filldraw[fill=white, draw=black, line width=0.4pt] (4,\row) circle (0.14);
    \draw[->, thick]
        (3.15,{\row+0.4}) -- (3.95,{\row+0.4});
    \node[below, font=\small] at (3.5,-0.4) {\name};
  \end{scope}
}

\end{tikzpicture}
\caption{Two distinct active recurrent classes on a two-dimensional torus with low density. In each panel, one horizontal row is fully occupied except for a single vacancy, and all other sites are vacant. The vacancy can move around its row, while vertical jumps are impossible. Each choice of row therefore gives a different active recurrent class with $N-1$ particles. Here $N=8$, with periodic boundary conditions.}
\label{fig:active-rows}
\end{figure}

More precisely, for an initial Bernoulli product distribution of density $\rho$, the expected picture is as follows.
\begin{enumerate}[label=(\arabic*)]
\item There is a \emph{structural critical density} $\rho^{\rm{st}}_{\star}\in(0,1)$ such that, for $\rho\in(\rho^{\rm{st}}_{\star},1)$, the particle-number sector selected by the initial configuration has a unique active recurrent class with high probability. On this class, one expects a hydrodynamic limit on the diffusive time scale.
\item For $\rho\in(0,\rho_\star^{\rm{st}})$, the particle-number sector selected by the initial configuration contains at least two recurrent classes with high probability. Let $\Sigma^{{\rm rec}}_{N}$ denote the union of all recurrent classes. One expects a further critical density $\rho_{\star}\in(0,\rho^{\rm{st}}_{\star})$ separating two dynamical regimes:
\begin{itemize}
\item \emph{Subcritical regime.} For $\rho\in(0,\rho_{\star})$, the hitting time of $\Sigma^{{\rm rec}}_{N}$ is poly-logarithmic in $N$ with high probability.
\item \emph{Supercritical regime.} For $\rho\in(\rho_{\star},\rho^{\rm{st}}_{\star})$, this hitting time is at least exponential in the system volume, i.e. at least $e^{cN^{d}}$ for some $c>0$, with high probability.
\end{itemize}
We refer to this conjectured change in the transience-time scale as a \emph{dynamical phase transition}, or an absorbing-state phase transition when the recurrent state reached is absorbing.
\end{enumerate}
For the directional FEP in dimensions $d\ge2$, the preceding picture has remained largely unproved. The present article establishes part of this structure and provides evidence for a dynamical transition strictly below the structural threshold.

\subsection{Our Achievements}
We present three groups of results that establish parts of the preceding picture and support its remaining conjectural features. Throughout this discussion, the dimension $d$ is at least two.

\subsubsection*{Main Result 1: Characterization of the Structural Critical Density}
Our first result, Theorem \ref{thm1}, identifies the structural critical density as
\[
\rho^{\rm{st}}_{\star}=1-2^{-d}.
\]
More precisely, when $N$ is even, we have a sharp threshold in the sense that every sector with $\rho^{\rm{st}}_{\star}N^{d}<K<N^{d}$ contains exactly one recurrent class, which is active, whereas every sector with $1\le K \le \rho_\star^{\rm st} N^d$ contains at least two recurrent classes. When $N$ is odd, we prove the same uniqueness for $N^d - ((N-1)/2)^d < K < N^d$ and the nonuniqueness for $1 \le K \le N^d - ((N+1)/2)^d$, which yield the same critical density $\rho_\star^{\rm st}$ as $N\to\infty$. This classification concerns recurrent classes and allows transient configurations outside them. Corollary \ref{cor:bernoulli-structural} gives the corresponding statement for the random sector selected by Bernoulli initial data.

\subsubsection*{Main Result 2: Poly-Logarithmic Hitting Time at Sufficiently Low Density}
We next study the approach to the recurrent set from Bernoulli product initial data. Define the \emph{transience time} by
\[
\tau_{{\rm tr}}:=\inf\bigl\{t\geq0:\eta^{N}_{t}\in\Sigma^{{\rm rec}}_{N}\bigr\}.
\]
Thus $\tau_{{\rm tr}}$ measures the duration of the transient regime; it is zero if the initial configuration is already recurrent. Write $\nu^{N}_{\rho}$ for the product measure under which sites are independently occupied with probability $\rho$. Theorem \ref{thm2} proves that there exists $\rho_{0}(d)\in(0,2^{-d})$ such that, for every fixed $\rho\in(0,\rho_{0}(d))$ and some constants $c,c'>0$,
\[
c\log N<\tau_{{\rm tr}}<c'(\log N)^{3d+1}
\]
with probability tending to one as $N\to\infty$. The configuration reached at $\tau_{{\rm tr}}$ is also absorbing with high probability. Hence a typical system at sufficiently low density freezes after a poly-logarithmic transient period. The upper bound and the high-probability absorption statement extend to initial distributions stochastically dominated by $\nu^{N}_{\rho}$ for some fixed $\rho\in(0,\rho_{0}(d))$; see Remark \ref{rem:thm2}.

\subsubsection*{Main Result 3: A Non-Poly-Logarithmic Hitting Time Lower Bound}
Theorem \ref{thm3} establishes a different behavior at intermediate densities. For every fixed $\rho\in(1/2,3/4)$,
\[
\tau_{{\rm tr}}>\frac{cN}{\log N}
\]
with probability tending to one as $N\to\infty$ along even integers. In particular, the transient regime lasts longer than $N^{\alpha}$ for every fixed $\alpha\in(0,1)$, with high probability. Since $3/4\leq\rho^{{\rm st}}_{\star}$ for every $d\geq2$, this entire interval lies strictly below the structural threshold. In two dimensions its upper endpoint is exactly $\rho^{{\rm st}}_{\star}=3/4$. Thus long transience occurs throughout an interval of fixed densities where recurrent classes still coexist. At these same densities, the one-dimensional process instead reaches its unique active recurrent class within a poly-logarithmic time \cite{BESS20}.

\subsubsection*{Ideas of the Proofs}
The proofs use three complementary geometric and probabilistic arguments. For the structural transition, an elementary counting argument shows that every configuration above density $1-2^{-d}$ contains a fully occupied elementary box with $2^{d}$ sites (see \eqref{eq:box-def} for its definition). Such boxes provide the facilitation needed for a reachability construction leading to a common recurrent class. Periodic vacancy constraints supply the obstructions to uniqueness below the threshold. This relates the structural critical density to the local geometry of the allowed jumps.

At low density, the main tool is an \emph{island decomposition}. An island is a region such that, along any legal evolution, particles cannot cross its boundary and neighboring sites on opposite sides of the boundary cannot both be occupied. Conditional on the initial configuration, the dynamics on distinct islands can therefore be treated independently. We prove that, for sufficiently small $\rho$, every minimal island has at most $c(\log N)^{d}$ sites with high probability. The key geometric step shows that a large island forces the presence of a large connected set with an unusually high initial particle density, an event excluded by counting connected sets and applying large deviation estimates. On each island, a spatial second-moment argument gives an expected absorption time bounded by the cube of its size, together with an exponential tail bound. Taking the maximum over the islands yields the upper bound $c(\log N)^{3d+1}$.

For the intermediate-density lower bound, we first establish a deterministic criterion for transience: for even $N$, if every coordinate line parallel to $\bm{e}_{1}$ has particle density strictly between $1/2$ and $3/4$, then the configuration is transient. The proof constructs a legal path, slice by slice, into the closed set of configurations in which all sites with even first and second coordinates are vacant. The initial line-density condition places the starting configuration outside this closed set, so the path proves that it cannot be recurrent. We then show that the stochastic evolution satisfies this criterion at time $cN/\log N$ with high probability. A graphical localization estimate controls the range of dependencies, while concentration of the resulting local variables keeps all line densities close to $\rho$. Together, these estimates turn the deterministic criterion into a lower bound on the transience time.

These results establish separated density regimes for rapid absorption and prolonged transience. Determining whether they are separated by a single sharp dynamical critical density remains an open problem. The optimal transience-time scales and the recurrent classes selected by Bernoulli initial data outside the low-density regime also remain to be understood. In the conjectured dynamical supercritical regime, we expect the transience time to be at least exponential in $N^d$, and more specifically conjecture the volume-exponential scale described above.

\section{\label{sec2}Main Results}

\begin{notation}
The following notation is assumed throughout the article.
\begin{itemize}
\item $c,c'>0$ are constants that are independent of $N$ and may vary line by line.

\item A statement holds \emph{with high probability} if its probability tends to $1$ as $N\to\infty$.

\item $a_{N}=O(b_{N})$ if $|a_{N}|\le cb_{N}$ for all $N\ge1$.

\item $a_{N}=\Theta(b_{N})$, or $a_{N}\asymp b_{N}$ if $a_{N}=O(b_{N})$ and $b_{N}=O(a_{N})$.

\item $a_{N}=o(b_{N})$ if $\lim_{N\to\infty}a_{N}/b_{N}=0$.

\item $a_{N}\simeq b_{N}$ if $\lim_{N\to\infty}a_{N}/b_{N}=1$.

\item $\lfloor\gamma\rfloor$ and $\lceil\gamma\rceil$ denote the greatest integer $\le\gamma$ and the least integer $\ge\gamma$, respectively.

\item $\llbracket\gamma,\gamma'\rrbracket:=[\gamma,\gamma']\cap\mathbb{Z}$, $\gamma\wedge\gamma':=\min\{\gamma,\gamma'\}$, and $\gamma\vee\gamma':=\max\{\gamma,\gamma'\}$.

\item $\mathbb{N}:=\{1,2,\dots\}$ and $\mathbb{N}_{0}:=\{0\}\cup\mathbb{N}$.
\end{itemize}
\end{notation}

\subsection{Model}
Consider a $d$-dimensional discrete torus of side length $N$, $\mathbb{T}^{d}_{N}=\{0,1,\dots,N-1\}^{d}$ for $d\ge1$. We represent each element as $\bm{x}=(x_{1},\dots,x_{d})\in\mathbb{T}^{d}_{N}$. We write $\bm{x}\sim\bm{y}$ if $\bm{x},\bm{y}\in\mathbb{T}^{d}_{N}$ are nearest neighbors in $\mathbb{T}^{d}_{N}$. Denote by $\bm{e}_{k}$ the $k$-th positive unit vector and collect all unit vectors as 
\begin{equation}
\mathcal{V}=\{\pm\bm{e}_{1},\dots,\pm\bm{e}_{d}\}.\label{eq:V-def}
\end{equation}
The configuration state space is defined as $\Sigma_{N}:=\{0,1\}^{\mathbb{T}^{d}_{N}}$ where the value $0$ (resp. $1$) indicates that the site is empty (resp. occupied). For $\eta=(\eta(\bm{x}))_{\bm{x}\in\mathbb{T}^{d}_{N}}\in\Sigma_{N}$ and $\bm{y},\bm{z}\in\mathbb{T}^{d}_{N}$, let $\eta^{\bm{y},\bm{z}}$ be obtained by exchanging the occupations at $\bm{y}$ and $\bm{z}$. The \emph{facilitated exclusion process} (FEP) $\{\eta^{N}_{t}\}_{t\ge0}$ is defined via an infinitesimal stochastic generator 
\begin{equation}
\mathcal{L}_{N}f(\eta)=\sum_{\bm{x}\in\mathbb{T}^{d}_{N}}\sum_{\bm{v}\in\mathcal{V}}\eta(\bm{x}-\bm{v})\eta(\bm{x})(1-\eta(\bm{x}+\bm{v}))\left(f(\eta^{\bm{x},\bm{x}+\bm{v}})-f(\eta)\right).\label{eq:FEP-def}
\end{equation}
Note that under the process the number of particles is conserved. In this regard, denote by $\Sigma_{N,K}$ the configuration space with $K\in\llbracket1,N^{d}-1\rrbracket$ particles: 
\[
\Sigma_{N,K}:=\{\eta\in\Sigma_{N}:|\eta|=K\},\qquad\text{where}\quad|\eta|:=\sum_{\bm{x}\in\mathbb{T}^{d}_{N}}\eta(\bm{x}).
\]
Here, we exclude the trivial cases $K=0$ and $K=N^{d}$. Throughout the article, an \emph{irreducible component} means a closed irreducible communicating class, or equivalently a recurrent class. Denote by $\mathbb{P}^{N}_{\mu}$ the law of the FEP trajectories with initial distribution $\mu$. If the process starts from a configuration $\eta\in\Sigma_{N}$, we write $\mathbb{P}^{N}_{\eta}$ instead of $\mathbb{P}^{N}_{\delta_{\eta}}$.

For our probabilistic results, we assume that the process starts from a Bernoulli product measure $\nu_{\rho}=\nu^{N}_{\rho}$ on $\Sigma_{N}$, which is defined as 
\begin{equation}
\nu^{N}_{\rho}:=\bigotimes_{\bm{x}\in\mathbb{T}^{d}_{N}}\nu_{\bm{x}}\qquad\text{with}\quad\nu_{\bm{x}}(1)=1-\nu_{\bm{x}}(0)=\rho.\label{eq:nu-rho-def}
\end{equation}
Here, $\rho\in(0,1)$ is a positive constant which controls the global particle density of the system. Indeed, a standard large deviations estimate implies that for any $\epsilon>0$, 
\[
\nu^{N}_{\rho}\left(\eta:\frac{|\eta|}{N^{d}}\notin(\rho-\epsilon,\rho+\epsilon)\right)\to0\qquad\text{exponentially as}\quad N\to\infty.
\]
Thus, according to $\nu^{N}_{\rho}$ the global density is concentrated near $\rho$.

\subsubsection*{Phase Transition for $d=1$}
It is well known \cite{BESS20} that the one-dimensional FEP undergoes the following phase transition at a critical density value $\rho_{\star}=1/2$.

\begin{itemize}
\item If $K\le N/2$, then the space $\Sigma_{N,K}$ is decomposed into $\Sigma^{{\rm tr}}_{N,K}$ and $\Sigma^{{\rm abs}}_{N,K}$, where 
\begin{align*}
\Sigma^{{\rm tr}}_{N,K} & =\{\eta\in\Sigma_{N,K}:\eta(x)=\eta(x+1)=1\quad\text{for some}\enspace x\in\mathbb{T}_{N}\},\\
\Sigma^{{\rm abs}}_{N,K} & =\{\eta\in\Sigma_{N,K}:\eta(x)\eta(x+1)=0\quad\text{for all}\enspace x\in\mathbb{T}_{N}\},
\end{align*}
such that each state in $\Sigma^{{\rm tr}}_{N,K}$ is transient, whereas each state in $\Sigma^{{\rm abs}}_{N,K}$ is absorbing.

\item If $K>N/2$, then $\Sigma_{N,K}$ is decomposed into $\Sigma^{{\rm tr}}_{N,K}$ and $\Sigma^{{\rm act}}_{N,K}$ where 
\begin{align*}
\Sigma^{{\rm tr}}_{N,K} & =\{\eta\in\Sigma_{N,K}:\eta(x)=\eta(x+1)=0\quad\text{for some}\enspace x\in\mathbb{T}_{N}\},\\
\Sigma^{{\rm act}}_{N,K} & =\{\eta\in\Sigma_{N,K}:\eta(x)+\eta(x+1)\ge1\quad\text{for all}\enspace x\in\mathbb{T}_{N}\},
\end{align*}
such that each state in $\Sigma^{{\rm tr}}_{N,K}$ is transient, whereas $\Sigma^{{\rm act}}_{N,K}$ is the unique active irreducible component. Here, we say that an irreducible component of a Markov chain is \emph{active} if it contains at least two elements.

\end{itemize}
Thus, the system undergoes a structural phase transition at $\rho_{\star}=1/2$. The critical value is determined by a representative configuration $\eta_{\star}\in\Sigma_{N,K}$ where $N$ is even and $K=N/2$, which is defined as 
\begin{equation}
\eta_{\star}(x)=\begin{cases}
1 & \text{if}\quad x\enspace\text{is even},\\
0 & \text{if}\quad x\enspace\text{is odd}.
\end{cases}\label{eq:ex-1D}
\end{equation}
Notice that no pairs of particles are adjacent in $\eta_{\star}$, thus $\eta_{\star}$ is an absorbing state.

The approach to the recurrent set also exhibits critical slowing down near the structural
threshold $1/2$. Starting from an initial Bernoulli $\nu^{N}_{\rho}$ with $\rho\in(0,1/2)$, the FEP system visits an absorbing state in a poly-logarithmic time scale of $\Theta(\log^{3}N)$ \cite{BEKL26}, and then stays frozen. If the initial density $\rho$ is close to criticality, i.e., if the initial distribution is uniform over all configurations whose constant particle density $\rho=\rho_N$ satisfies $1/2-\rho\simeq N^{-a}$ where $a\in(0,1)$ is a fixed constant, then the absorption time lies in $(N^{2\wedge(4a)-\epsilon},N^{2\wedge(4a)+\epsilon})$ with high probability for any fixed $\epsilon>0$ \cite{BEKL26}, thereby exhibiting a polynomial scale. Finally, if the system starts from Bernoulli $\nu^{N}_{\rho}$ with $\rho\in(1/2,1)$, it enters the single ergodic phase $\Sigma^{{\rm act}}_{N,K}$ within a poly-logarithmic time $O(\log^{32}N)$ \cite{BESS20}, and then converges to the stationary state of $\Sigma^{{\rm act}}_{N,K}$ as time diverges by a hydrodynamic mechanism. For further details on this topic, see \cite{BESS20,BES21,DCES26} and the references therein.

\subsubsection*{Phase Transitions for $d\ge2$}
The main objective of this article is to identify the structural critical density in dimensions two or higher and to establish distinct dynamical regimes below it. Our results support the conjecture that the structural and dynamical transitions occur at \emph{different} critical densities. The geometric mechanisms underlying this picture can be illustrated by the following two examples $\eta_{\star1},\eta_{\star2}\in\Sigma_{N}=\{0,1\}^{\mathbb{T}^{d}_{N}}$ for even $N$, which generalize $\eta_{\star}$ in \eqref{eq:ex-1D} in two different ways. Define 
\begin{equation}
\eta_{\star1}(\bm{x})=\begin{cases}
1 & \text{if at least one of}\enspace x_{1},\dots,x_{d}\enspace\text{is even},\\
0 & \text{if}\quad x_{1},\dots,x_{d}\enspace\text{are all odd},
\end{cases}\quad\eta_{\star2}(\bm{x})=\begin{cases}
1 & \text{if}\quad x_{1}+\cdots+x_{d}\enspace\text{is even},\\
0 & \text{if}\quad x_{1}+\cdots+x_{d}\enspace\text{is odd}.
\end{cases}\label{eq:ex-dD}
\end{equation}
See the first (for $\eta_{\star1}$) and the second (for $\eta_{\star2}$) configurations in Figure \ref{fig1} for $d=2$ and $N=12$. It is easy to verify that $\eta_{\star1},\eta_{\star2}$ are absorbing states; however, the reasons why they are absorbing are completely different. The configuration $\eta_{\star2}$ is absorbing since there are no pairs of neighboring particles, whereas $\eta_{\star1}$ is absorbing due to the following reason. If there were an allowed particle jump in $\eta_{\star1}$ as $\bm{x}\to\bm{x}+\bm{v}$ for $\bm{x}\in\mathbb{T}^{d}_{N}$ and $\bm{v}\in\mathcal{V}$, by definition we have $\eta_{\star1}(\bm{x}+\bm{v})=0$, thus all $x_{i}+v_{i}$, $i\in\llbracket1,d\rrbracket$, must be odd. This implies that all $x_{i}-v_{i}$ are also odd, thus $\eta_{\star1}(\bm{x}-\bm{v})=0$, which contradicts the fact that the jump $\bm{x}\to\bm{x}+\bm{v}$ is admissible.

Notice that the density of $\eta_{\star1}$ is $1-2^{-d}$, whereas the density of $\eta_{\star2}$ is $1/2$. As one can guess from the third and fourth configurations in Figure \ref{fig1}, these distinct mechanisms may coexist and thereby produce a highly complicated structure of absorbing states.

Another interesting point is the existence of non-unique active recurrent classes—a phenomenon not observed in one dimension—which can be explained by the last configuration in Figure \ref{fig1}. Here, the configuration is not absorbing since the empty site (marked with $\star$) on the region of gray rows and columns can be filled up by a legal jump of a gray particle. However, this empty site can visit all sites in the gray region but cannot escape it, thus the collection of all such configurations, with the same gray region and red particles with a single empty site on the gray region, becomes an irreducible component. It is clear that this irreducible component is not unique at the same density level, since one can translate the whole component by $+\bm{e}_{1}$ to obtain another irreducible class.

\begin{figure}[htbp]
\centering
\begin{tikzpicture}[scale=0.30,
  graydot/.style={draw=black,fill=black!30!white,line width=0.35pt},
  reddot/.style={draw=black,fill=red!50!white,line width=0.35pt}]
\foreach \dx/\dy/\name in {
  0/0/{(a) $\eta_{\star1}$},14/0/{(b) $\eta_{\star2}$},
  28/0/{(c)},7/-15/{(d)},21/-15/{(e)}} {
  \begin{scope}[shift={(\dx,\dy)}]
    \draw[gray!45,line width=0.25pt] (0,0) grid (12,12);
    \node[below,font=\small] at (6,-0.6) {\name};
  \end{scope}
}
\begin{scope}
\foreach \i in {0,2,4,6,8,10} {\foreach \j in {0,...,11}
  \filldraw[graydot] (\i+0.5,\j+0.5) circle (0.3);}
\foreach \i in {1,3,5,7,9,11} {\foreach \j in {0,2,4,6,8,10}
  \filldraw[graydot] (\i+0.5,\j+0.5) circle (0.3);}
\end{scope}
\begin{scope}[shift={(14,0)}]
\foreach \i in {0,2,4,6,8,10} {\foreach \j in {0,2,4,6,8,10}
  \filldraw[reddot] (\i+0.5,\j+0.5) circle (0.3);}
\foreach \i in {1,3,5,7,9,11} {\foreach \j in {1,3,5,7,9,11}
  \filldraw[reddot] (\i+0.5,\j+0.5) circle (0.3);}
\end{scope}
\begin{scope}[shift={(28,0)}]
\foreach \i in {0,4,10} {\foreach \j in {0,...,11}
  \filldraw[graydot] (\i+0.5,\j+0.5) circle (0.3);}
\foreach \j in {2,9} {\foreach \i in {0,...,11}
  \filldraw[graydot] (\i+0.5,\j+0.5) circle (0.3);}
\foreach \i/\j in {2/0,2/5,2/7,6/0,6/4,7/6,8/5,8/7,8/11}
  \filldraw[reddot] (\i+0.5,\j+0.5) circle (0.3);
\end{scope}
\begin{scope}[shift={(7,-15)}]
\foreach \j in {0,4,10} {\foreach \i in {0,...,11}
  \filldraw[graydot] (\i+0.5,\j+0.5) circle (0.3);}
\foreach \i/\j in {2/2,7/2,10/2,1/6,4/6,6/6,9/6,
  2/7,5/7,11/7,0/8,4/8,7/8,9/8}
  \filldraw[reddot] (\i+0.5,\j+0.5) circle (0.3);
\end{scope}
\begin{scope}[shift={(21,-15)}]
\foreach \i in {0,1,2,4,5,6,7,8,9,11} {\foreach \j in {3,10}
  \filldraw[graydot] (\i+0.5,\j+0.5) circle (0.3);}
\foreach \j in {0,1,2,4,5,6,7,8,9,10,11}
  \filldraw[graydot] (3.5,\j+0.5) circle (0.3);
\foreach \j in {0,...,11}
  \filldraw[graydot] (10.5,\j+0.5) circle (0.3);
\foreach \i/\j in {0/8,1/5,5/7,6/1,7/5,8/0,8/6,8/8}
  \filldraw[reddot] (\i+0.5,\j+0.5) circle (0.3);
\node[font=\small,inner sep=0pt] at (3.5,3.5) {$\star$};
\end{scope}
\end{tikzpicture}
\caption{Configurations $\eta_{\star1}$ and $\eta_{\star2}$,
  two examples of absorbing configurations combining the two mechanisms, and a configuration
  belonging to an active irreducible component, respectively.}
\label{fig1}
\end{figure}
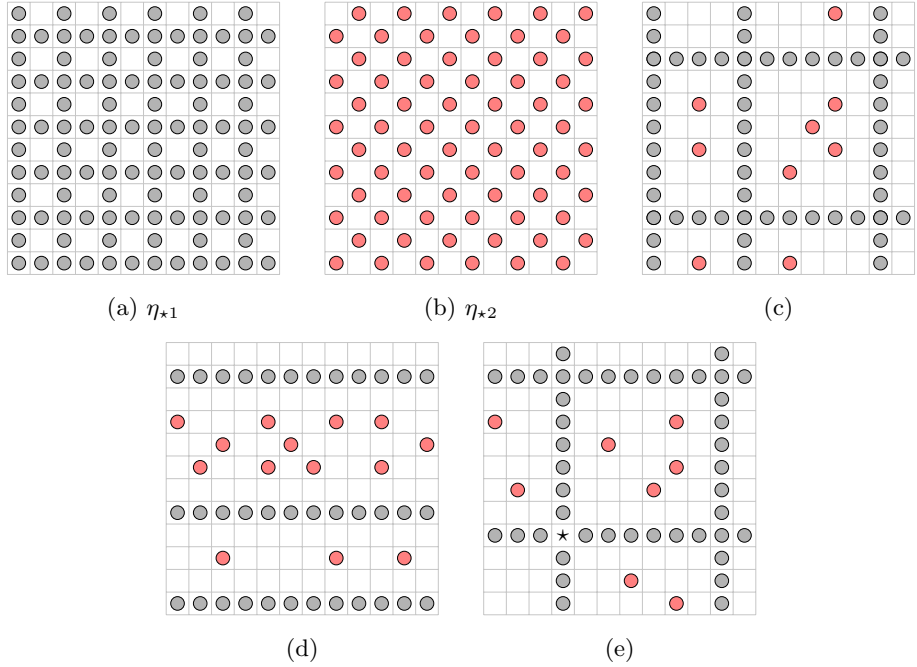

\subsection{Structural Phase Transition for Higher-Dimensional FEPs}
Our first main result states that a sharp structural phase transition occurs at the criticality  
\begin{equation}
\rho^{{\rm st}}_{\star}:=1-\frac{1}{2^{d}}.\label{eq:rho-star-st}
\end{equation}

\begin{thm}
\label{thm1}Suppose that $d\ge2$ and let $K$ and $N$ be positive integers such that $1\le K<N^{d}$.
\begin{enumerate}
\item If $K$ and $N$ satisfy 
\[
K > \begin{cases}
\rho^{{\rm st}}_{\star}N^{d} & \text{if}\quad N\enspace\text{is even},\\
N^{d}-\left(\frac{N-1}{2}\right)^{d} & \text{if}\quad N\enspace\text{is odd},
\end{cases}
\]
then there is exactly one irreducible component in $\Sigma_{N,K}$, and it is active.
\item If $K$ and $N$ satisfy 
\[
K\le\begin{cases}
\rho^{{\rm st}}_{\star}N^{d} & \text{if}\quad N\enspace\text{is even},\\
N^{d}-\left(\frac{N+1}{2}\right)^{d} & \text{if}\quad N\enspace\text{is odd},
\end{cases}
\]
then $\Sigma_{N,K}$ contains at least two irreducible components.
\end{enumerate}
\end{thm}

The previous result implies that we have the following structural phase transition for the Bernoulli initial data. 

\begin{cor} 
\label{cor:bernoulli-structural}
Fix $d\ge2$ and $\rho\in(0,1)$. Let $\eta_0\sim\nu^{N}_{\rho}$ and set $K_N:=|\eta_0|$.
\begin{enumerate}[label=(\arabic*)]
\item If $\rho>\rho^{{\rm st}}_{\star}$, with high probability, the sector $\Sigma_{N,K_N}$ has exactly one recurrent class, which is active.
\item If $\rho<\rho^{{\rm st}}_{\star}$, with high probability, the sector $\Sigma_{N,K_N}$ has at least two recurrent classes.
\end{enumerate}
\end{cor}
\begin{proof}
Under $\nu^{N}_{\rho}$, the random variable $K_N$ has the binomial distribution with parameters $N^{d}$ and $\rho$. In particular, $K_N/N^{d}\to\rho$ in probability, and $1\le K_N<N^{d}$ with probability tending to one. For $\rho>\rho^{{\rm st}}_{\star}$, the first conclusion follows directly from Theorem \ref{thm1}-(1), along with the fact that for odd $N$,
\[
1- \left( \frac{N-1}{2N} \right)^d \xrightarrow{N \to \infty} 1- 2^{-d} = \rho_\star^{\rm st}.
\]
For $\rho<\rho^{{\rm st}}_{\star}$, the normalized upper threshold in Theorem \ref{thm1}-(2) equals $\rho^{{\rm st}}_{\star}$ for even $N$, while for odd $N$ it is
\[
1-\left(\frac{N+1}{2N}\right)^{d} \xrightarrow{N \to \infty} 1-2^{-d}=\rho^{{\rm st}}_{\star}.
\]
The same concentration estimate therefore allows us to apply Theorem \ref{thm1}-(2) for both parities.
\end{proof}
The corollary describes the recurrent structure of the random particle-number sector selected by the initial configuration. It does not assert that the initial configuration is recurrent, nor does it determine which recurrent class the dynamics eventually enters.

\subsection{Dynamical Regimes for Higher-Dimensional FEPs}
As our second and third main results, we reveal two distinct dynamical regimes below the structural critical density $\rho^{{\rm st}}_{\star}=1-2^{-d}$. For Bernoulli initial data at sufficiently low densities, the process is absorbed within a poly-logarithmic time with high probability. In contrast, for every $\rho\in(\frac{1}{2},\frac{3}{4})$ and even $N$, the transience time exceeds $cN/\log N$ with high probability.

To see this, let us denote by $\Sigma^{{\rm rec}}_{N}$ the collection of all recurrent states and by $\Sigma^{{\rm tr}}_{N}$ the collection of all transient states, such that $\Sigma_{N}=\Sigma^{{\rm rec}}_{N}\cup\Sigma^{{\rm tr}}_{N}$. Then, further decompose $\Sigma^{{\rm rec}}_{N}$ into 
\[
\Sigma^{{\rm rec}}_{N}=\Sigma^{{\rm abs}}_{N}\cup\Sigma^{{\rm act}}_{N},
\]
where $\Sigma^{{\rm abs}}_{N}$ collects the absorbing states and $\Sigma^{{\rm act}}_{N}$ is the union of active irreducible components. We denote by $\tau_{{\rm tr}}=\tau^{N}_{{\rm tr}}$ the first hitting time of the recurrent set $\Sigma^{{\rm rec}}_{N}$, i.e., 
\[
\tau_{{\rm tr}}=\tau^{N}_{{\rm tr}}:=\inf\left\{ t\ge0:\eta^{N}_{t}\in\Sigma^{{\rm rec}}_{N}\right\} .
\]
We call $\tau_{{\rm tr}}$ the \emph{transience time} of the system. Mind that $\tau_{{\rm tr}}=0$ if the system already starts from a configuration in $\Sigma^{{\rm rec}}_{N}$.

First, the transience time $\tau_{{\rm tr}}$ is poly-logarithmic if the initial density $\rho$ is sufficiently small. Recall from \eqref{eq:nu-rho-def} that $\nu^{N}_{\rho}$ denotes the Bernoulli product measure on $\Sigma_{N}$.

\begin{thm}
\label{thm2}For $d\ge2$, there exists $\rho_{0}=\rho_{0}(d)\in(0,2^{-d})$ such that for all $\rho\in(0,\rho_{0})$, 
\[
\lim_{N\to\infty}\mathbb{P}^{N}_{\nu_{\rho}}\left[c\log N<\tau_{{\rm tr}}<c'(\log N)^{3d+1},\quad\eta^{N}_{\tau_{{\rm tr}}}\in\Sigma^{{\rm abs}}_{N}\right]=1,
\]
for some constants $c=c(\rho,d),c'=c'(\rho,d)>0$. In words, starting from initial density $\rho\in(0,\rho_{0})$, the process gets \emph{absorbed} within a poly-logarithmic time with high probability.
\end{thm}

\begin{remark}[Stochastically dominated initial distributions]\label{rem:thm2} Careful reading of our argument reveals that the upper bound and absorption conclusion of Theorem \ref{thm2} also hold for any sequence of initial distributions $\mu_N$ stochastically dominated by $\nu^{N}_{\rho}$ for some fixed $\rho\in(0,\rho_0(d))$. Here stochastic domination refers to the coordinatewise order on configurations. More precisely, for a constant $C=C(\rho,d)>0$,
\[
\lim_{N\to\infty}\mathbb{P}^{N}_{\mu_N}\left[
\tau_{{\rm tr}}<C(\log N)^{3d+1},\quad
\eta^{N}_{\tau_{{\rm tr}}}\in\Sigma^{{\rm abs}}_{N}
\right]=1.
\]
Indeed, the event that there exists a large connected set with the initial density bound used in Lemma \ref{l3} is increasing in the initial configuration. Its probability under $\mu_N$ is therefore no larger than under $\nu^{N}_{\rho}$. A minimal island containing no particles is a singleton. The deterministic island decomposition and the uniform absorption estimates in Section \ref{sec4} therefore give the conclusion without a positive lower bound on the initial density. This argument uses domination of the initial distributions, rather than a monotone coupling of the evolutions. The logarithmic lower bound does not extend under stochastic domination alone: the point mass at the empty configuration has $\tau_{{\rm tr}}=0$.
\end{remark}

Next, if the initial density $\rho$ lies in between $1/2$ and $3/4\le\rho^{{\rm st}}_{\star}=1-2^{-d}$, the transience time is at least polynomial, under the restriction that $N$ is even.

\begin{thm}
\label{thm3}Suppose that $d\ge2$. For all $\rho\in(1/2,3/4)$, 
\[
\lim_{\substack{N\to\infty\\
N\ \text{is even}
}
}\mathbb{P}^{N}_{\nu_{\rho}}\left[\tau_{{\rm tr}}>\frac{cN}{\log N}\right]=1,
\]
for some constant $c=c(\rho,d)>0$. In particular, for every $\alpha\in(0,1)$, the transience time exceeds $N^{\alpha}$ with high probability as $N \to \infty$ through even values.
\end{thm}

\begin{remark}
We conjecture that, for every $d \ge 2$ and $\rho \in (1/2 , \rho_\star^{\rm st} )$, there exists a constant $c=c(d,\rho)>0$ such that
\begin{equation}\label{eq:conjecture}
\lim_{N\to\infty} \mathbb{P}^{N}_{\nu_{\rho}}\left[\tau_{{\rm tr}} > e^{cN^d} \right]=1.
\end{equation}
Together with Theorem \ref{thm2}, this would establish a stronger separation between the two
dynamical regimes: poly-logarithmic transience times at sufficiently low densities and transience times at
least exponential in the volume $N^d$ throughout $(1/2, \rho_\star^{\rm st} )$. In this connection, we discuss three potential improvements of Theorem \ref{thm3}.
\begin{itemize}
\item In the intermediate regime $\rho\in(1/2,\rho_\star^{\rm st})$, we expect the process started from $\nu_\rho$ to typically reach a recurrent configuration containing on the order of $N$ fully occupied $(d-1)$-dimensional submanifolds, each of size $N^{d-1}$. Informal simulations for $d = 2$ and $N = 100$ suggest this geometric picture in two dimensions. We expect that forming such a global structure requires exponentially long times, which motivates \eqref{eq:conjecture}. Our proof, however, relies on a local (de)localization coupling (cf. Lemma \ref{lem:plb-localization}) that gives a meaningful bound only up to polynomial time scales. Establishing the conjectured exponential lower bound would therefore require a better understanding of the global transition mechanism from $\nu_\rho$ to such recurrent configurations.
\item For $d\ge3$, we establish the polynomial lower bound only on the smaller interval $(1/2,3/4)\subset(1/2,\rho_\star^{\rm st})$. This restriction arises because the transient path constructed in Lemma \ref{lem:plb-transient} relies on ergodicity properties of the dynamics on one-dimensional submanifolds. A better understanding of the recurrent/transient structure of the $(d-1)$-dimensional FEP could allow us to extend the polynomial lower bound throughout the intermediate regime $(1/2,\rho_\star^{\rm st})$.
\item The evenness assumption in Theorem \ref{thm3} arises from the present construction: within each
two-dimensional slice, the proof pairs alternating even and odd rows as donors and receivers. This
pairing requires $N$ to be even. We expect an analogous criterion, and hence the conclusion of Theorem \ref{thm3}, to hold for odd $N$ as well. Establishing such a criterion poses a separate combinatorial problem, which we do not pursue here.
\end{itemize}

\end{remark}

\subsubsection*{Organization of the Paper}
The remainder of the paper is devoted to the proofs of the three main results. Section \ref{sec3} proves Theorem \ref{thm1} by constructing recurrent classes below the structural threshold and establishing uniqueness above it. Section \ref{sec4} develops the island decomposition and proves the low-density absorption estimates of Theorem \ref{thm2}. Section \ref{sec5} proves Theorem \ref{thm3} by combining a deterministic transience criterion with graphical localization and concentration of line densities.

\section{\label{sec3}Structural Phase Transition and Proof of Theorem \ref{thm1}}

\subsection{\label{sec3.1}Above the Structural Threshold: A Unique Recurrent Class}
In this subsection, we prove part (1) of Theorem \ref{thm1}.

For $\bm{x}\in\mathbb{T}^{d}_{N}$, let us call 
\begin{equation}
\mathcal{B}_{\bm{x}}:=\prod^{d}_{i=1}\{x_{i},x_{i}+1\}\label{eq:box-def}
\end{equation}
the elementary box anchored at $\bm{x}$. Note that $|\mathcal{B}_{\bm{x}}|=2^{d}$. We start with a lemma.

\begin{lem}
\label{lem3.2}Recall \eqref{eq:rho-star-st}. If $K>\rho^{{\rm st}}_{\star}N^{d}$, then for any $\eta\in\Sigma_{N,K}$ there exists $\bm{x}\in\mathbb{T}^{d}_{N}$ such that $\eta=1$ identically on $\mathcal{B}_{\bm{x}}$.
\end{lem}

\begin{proof}
Fix $\eta\in\Sigma_{N,K}$ and suppose the contrary that $\eta$ is not identically $1$ on any $\mathcal{B}_{\bm{x}}$, $\bm{x}\in\mathbb{T}^{d}_{N}$. This implies that 
\[
2^{d}-1\ge\sum_{\bm{y}\in\mathcal{B}_{\bm{x}}}\eta(\bm{y})\qquad\text{for every}\quad\bm{x}\in\mathbb{T}^{d}_{N}.
\]
 Summing over all $N^d$ elementary boxes and using the fact that each site belongs to exactly $2^d$ such boxes, we obtain the following inequality.
\[
N^{d}(2^{d}-1)=\sum_{\bm{x}\in\mathbb{T}^{d}_{N}}(2^{d}-1)\ge\sum_{\bm{x}\in\mathbb{T}^{d}_{N}}\sum_{\bm{y}\in\mathcal{B}_{\bm{x}}}\eta(\bm{y})=\sum_{\bm{y}\in\mathbb{T}^{d}_{N}}2^{d}\eta(\bm{y})=2^{d}K,
\]
thus we conclude that $K\le\rho^{{\rm st}}_{\star}N^{d}$, which contradicts the assumption.
\end{proof}
First, assume that $N$ is even. Define
\[
\mathbb{O}^{d}_{N}  :=\left\{ \bm{x}\in\mathbb{T}^{d}_{N}:x_{1},x_{2},\dots,x_{d}\enspace\text{are all odd}\right\},
\qquad \mathbb{E}^{d}_{N}  := \mathbb T_N^d \setminus \mathbb O_N^d .
\]
Let us write $\eta\rightsquigarrow\xi$ if $\xi$ is reachable from $\eta$ by a finite number of legal jumps.

\begin{lem}
\label{lem3.3}For any $\eta\in\Sigma_{N,K}$ where $N$ is even and $K>\rho^{{\rm st}}_{\star}N^{d}$, there exists $\xi\in\Sigma_{N,K}$ with $\eta\rightsquigarrow\xi$ such that $\xi=1$ identically on $\mathbb{E}^{d}_{N}$.
\end{lem}

\begin{proof}
Fix $K>\rho^{{\rm st}}_{\star}N^{d}$ and $\eta\in\Sigma_{N,K}$. Suppose the contrary, such that starting from $\eta$ we cannot fully occupy the set $\mathbb{E}^{d}_{N}$. Let $\xi\in\Sigma_{N,K}$ be a configuration with $\eta\rightsquigarrow\xi$ which has the maximum number of particles on $\mathbb{E}^{d}_{N}$. By the assumption, $\xi$ has at least one empty site in $\mathbb{E}^{d}_{N}$, which we will denote by $\bm{o}\in\mathbb{E}^{d}_{N}$.

It now suffices to find another $\zeta\in\Sigma_{N,K}$ with $\xi\rightsquigarrow\zeta$ such that $\zeta$ has one more particle on $\mathbb{E}^{d}_{N}$ than $\xi$, which contradicts the maximality assumption and thus concludes the proof. The idea is to apply Lemma \ref{lem3.2} to move a fully occupied elementary box toward the target site $\bm{o}$ until it reaches a site adjacent to $\bm{o}$. A final finite sequence of legal jumps then moves a particle to $\bm{o}$, which produces a new configuration with one more particle on $\mathbb{E}^{d}_{N}$.

More precisely, we start from $\xi_{0}=\xi$ and denote by $\mathcal{B}$ a fully occupied elementary box in $\mathbb{T}^{d}_{N}$, which exists by Lemma \ref{lem3.2} applied to $\xi$. There exists a sequence of elementary boxes $\mathcal{B}=\mathcal{B}_{0},\mathcal{B}_{1},\dots,\mathcal{B}_{m}$ such that each pair $(\mathcal{B}_{i},\mathcal{B}_{i+1})$ shares exactly $2^{d-1}$ common sites and $\mathcal{B}_{m}$ is adjacent to $\bm{o}$. Consider each step $i\in\llbracket0,m-1\rrbracket$ with a configuration $\xi_{i}$ and a fully occupied box $\mathcal{B}_{i}$ of $\xi_{i}$. There are $d2^{d}$ neighboring sites of $\mathcal{B}_{i}$, and exactly $d$ of them belong to $\mathbb{O}^{d}_{N}$ while the others belong to $\mathbb{E}^{d}_{N}$. Let us denote them by $\mathcal{O}_{i}$ and $\mathcal{E}_{i}$, respectively.

First, suppose that a site in $\mathcal{E}_{i}$ is empty. Then, let $\bm b$ be the only element of $\mathcal B_i \cap \mathbb O_N^d$, and define
\[
\ell(\bm u) := | \{ j \in \llbracket 1,d \rrbracket : u_j \ne b_j \} | \qquad \text{for} \quad \bm u \in \mathcal B_i.
\]
Then, take a pair of empty $\bm y \in \mathcal E_i$ and its neighbor $\bm u \in \mathcal B_i$ such that $\ell := \ell (\bm u)$ is minimal. By definition, there exists a nearest-neighbor path of sites $\bm u = \bm u_0 , \bm u_1 , \dots , \bm u_\ell = \bm b$ in $\mathcal B_i$ such that $\ell(\bm u_j) = \ell(\bm u) - j$. Then, we may move the empty position, initially at $\bm y$, along the path $\bm y \to \bm u \to \bm u_1 \to \cdots \to \bm u_\ell = \bm b$, which is legitimate since each facilitator position $2\bm u_j - \bm u_{j-1}$, $j \ge 1$, is occupied by the minimality assumption of $\bm u$, whereas the first facilitator position $2 \bm u - \bm y$ is inside $\mathcal B_i$ so it is occupied. This contradicts the maximality of $\xi$ since we have moved a particle at $\bm b \in \mathbb O_N^d$ to $\bm y$. Thus, all sites in $\mathcal{E}_{i}$ must be occupied.

Next, consider the set $\mathcal{B}_{i+1}\setminus\mathcal{B}_{i}$ of size $2^{d-1}$. If a site in $(\mathcal{B}_{i+1}\setminus\mathcal{B}_{i})\cap\mathcal{O}_{i}$ is empty, then the three consecutive sites immediately behind it are all occupied: the first belongs to $(\mathcal{B}_{i}\cap\mathcal{B}_{i+1})\cap\mathbb{E}^{d}_{N}$, the second belongs to $(\mathcal{B}_{i}\setminus\mathcal{B}_{i+1})\cap\mathbb{O}^{d}_{N}$, and the third belongs to $\mathcal{E}_{i}$. This means that we may push the particles two times in a row to eventually move the particle at $(\mathcal{B}_{i}\setminus\mathcal{B}_{i+1})\cap\mathbb{O}^{d}_{N}$ to $(\mathcal{B}_{i+1}\setminus\mathcal{B}_{i})\cap\mathcal{O}_{i}$, maintaining the total number of particles on $\mathbb{E}^{d}_{N}$. This procedure can be carried out to all empty sites in $(\mathcal{B}_{i+1}\setminus\mathcal{B}_{i})\cap\mathcal{O}_{i}$, thereby making them all occupied. After all these procedures, we obtain a new configuration $\xi_{i+1}$ with $\xi_{i}\rightsquigarrow\xi_{i+1}$ which has the same number of particles on $\mathbb{E}^{d}_{N}$, has $\mathcal{B}_{i+1}$ fully occupied, and leaves the occupation of every site in $\mathbb E_N^d$ unchanged.

Repeating this algorithm for all $i\in\llbracket0,m-1\rrbracket$, we obtain a terminal configuration $\xi_{m}$ and a fully occupied box $\mathcal{B}_{m}$ of $\xi_{m}$ which is adjacent to $\bm{o}$.
The configuration $\xi_m$ still maximizes the number of particles on $\mathbb E_N^d$. By the preceding maximality argument, every site in $\mathbb E_N^d$ adjacent to the fully occupied box $\mathcal B_m$, including $\bm o$, must be occupied. However, the construction leaves the occupation of every site in $\mathbb E_N^d$ unchanged, so $\bm o$ remains empty. This contradiction completes the proof.
\end{proof}
Next, define 
\[
\Gamma_{N,K}:=\left\{ \xi\in\Sigma_{N,K}:\xi=1\quad\text{identically on}\enspace\mathbb{E}^{d}_{N}\right\} .
\]

\begin{lem}
\label{lem3.4}For any $\xi,\xi'\in\Gamma_{N,K}$ where $N$ is even, we have $\xi\rightsquigarrow\xi'$.
\end{lem}

\begin{proof}
First, let us take two configurations $\xi,\xi'\in\Gamma_{N,K}$ such that there exist $\bm{x}\in\mathbb{O}^{d}_{N}$ and $\bm{v}\in\mathcal{V}$ such that $\xi(\bm{x})=\xi'(\bm{x}+2\bm{v})=1$, $\xi(\bm{x}+2\bm{v})=\xi'(\bm{x})=0$, and they agree on all other sites. Then, since $\bm{x},\bm{x}+2\bm{v}\in\mathbb{O}^{d}_{N}$ and $\bm{x}+\bm{v},\bm{x}-\bm{v}\in\mathbb{E}^{d}_{N}$, we have 
\[
\xi(\bm{x}-\bm{v})=\xi(\bm{x})=\xi(\bm{x}+\bm{v})=1\qquad\text{and}\quad\xi(\bm{x}+2\bm{v})=0.
\]
Thus by two jumps in a row, we can move the particle at $\bm{x}$ to $\bm{x}+2\bm{v}$, thereby obtaining the configuration $\xi'$. Thus, $\xi\rightsquigarrow\xi'$.

Now, for the configurations in $\Gamma_{N,K}$, the $N^{d}-K$ empty sites are elements of $\mathbb{O}^{d}_{N}$. Let us identify $\mathbb{O}^{d}_{N}$ with the half-size periodic torus $\mathbb{T}^{d}_{N/2}$ by declaring that 
\[
\bm{x}\sim'\bm{y}\qquad\text{if and only if}\qquad\frac{\bm{x}-\bm 1}{2}\sim\frac{\bm{y}-\bm 1}{2}\qquad\text{in}\quad\mathbb{T}^{d}_{N/2}.
\]
The claim in the previous paragraph indicates that any particle in $\mathbb{O}^{d}_{N}$ can move to its neighbor (with respect to $\sim'$) in $\mathbb{O}^{d}_{N}$ by legal jumps. Thus, these allowed transitions between configurations in $\Gamma_{N,K}$ can be identified with exclusion-rule nearest-neighbor jumps on the corresponding space $\mathbb{T}^{d}_{N/2}$. Since it is clear that the set $\Sigma_{N/2,N^d-K}$ is irreducible with respect to the simple exclusion process, we have proved that for any given $\xi,\xi'\in\Gamma_{N,K}$ there exists a series of legal jumps from $\xi$ to $\xi'$, thus $\xi\rightsquigarrow\xi'$.
\end{proof}

\begin{proof}
[Proof of Theorem \ref{thm1}-(1) when $N$ is even] Denote by $\mathfrak{J}_{N,K}$ the irreducible class that contains all elements in $\Gamma_{N,K}$, which exists by Lemma \ref{lem3.4}. Now, take any recurrent element $\eta\in\Sigma_{N,K}$. By Lemma \ref{lem3.3}, there exists a series of legal jumps from $\eta$ to the irreducible class $\mathfrak{J}_{N,K}$, which automatically verifies that $\eta\in\mathfrak{J}_{N,K}$. Thus, $\mathfrak{J}_{N,K}$ is exactly the unique irreducible component of $\Sigma_{N,K}$. In addition, $|\mathfrak{J}_{N,K}|\ge|\Gamma_{N,K}|\ge2$, so $\mathfrak{J}_{N,K}$ is clearly active, which completes the proof if $N$ is even.
\end{proof}

Next, let us assume that $N$ is odd. The overall strategy is the same, and the suboptimal conditions in Theorem \ref{thm1} come from the boundary behavior and parity disagreements.
We continue to use the same sets $\mathbb O_N^d$, $\mathbb E_N^d$, and $\Gamma_{N,K}$ as defined above.

\begin{lem}
\label{lem3.2-odd}Suppose that $N$ is odd. If $K> N^d - ((N-1)/2)^d$, then for any $\eta\in\Sigma_{N,K}$ there exists $\bm{x}\in\{0,1,\dots,N-2\}^d$ such that $\eta=1$ identically on $\mathcal{B}_{\bm{x}}$.
\end{lem}

\begin{proof}
Let $m=(N-1)/2$. Notice that the $m^d$ elementary boxes anchored at $2\bm y \in \mathbb T_N^d$, $\bm y \in \{0,1,\dots,(N-3)/2\}^d$, are mutually disjoint. Since there are fewer than $m^d$ empty sites, at least one of them should be fully occupied as desired.
\end{proof}

\begin{lem}
\label{lem3.3-odd}For any $\eta\in\Sigma_{N,K}$ where $N$ is odd and $K> N^d - ((N-1)/2)^d$, there exists $\xi\in\Sigma_{N,K}$ with $\eta\rightsquigarrow\xi$ such that $\xi=1$ identically on $\mathbb{E}^{d}_{N}$.
\end{lem}

\begin{proof}
As done in the proof of Lemma \ref{lem3.3}, suppose the contrary and take $\xi \in \Sigma_{N,K}$ with $\eta \rightsquigarrow \xi$ such that $\xi$ has the maximum number of particles on $\mathbb E_N^d$, which has at least one empty site $\bm o \in \mathbb E_N^d$. Then, we employ the same idea and construct a sequence of boxes $\mathcal B_0,\mathcal B_1, \dots, \mathcal B_m$ from an initial full box to a box adjacent to $\bm o$. We may take this sequence to avoid crossing the periodic boundary. Then by the same logic, at each step $i$, any site in $\mathcal E_i$ must be already occupied, since otherwise it violates the maximality of $\xi$. Moreover, the idea of occupying all sites in $(\mathcal B_{i+1} \setminus \mathcal B_i) \cap \mathcal O_i$ works as well; here we use the fact that $\mathbb O_N^d$ is away from the boundary of $\mathbb T_N^d$; thus, the third facilitator does not go across the boundary. Thus, we may push this box to the terminal box and complete the proof exactly as in Lemma \ref{lem3.3}.
\end{proof}

\begin{lem}
\label{lem3.4-odd}For any $\xi,\xi'\in\Gamma_{N,K}$ where $N$ is odd, we have $\xi\rightsquigarrow\xi'$.
\end{lem}

\begin{proof}
The same idea as in the proof of Lemma \ref{lem3.4} applies here as well; the only difference is that to prevent any movement across the boundary, we identify $\mathbb O_N^d$ with the lattice box $\llbracket 0,(N-3)/2 \rrbracket^d$ with \emph{open} rather than periodic boundary conditions. Since the simple exclusion process on this box is also irreducible, the same argument applies.
\end{proof}

\begin{proof}[Proof of Theorem \ref{thm1}-(1) when $N$ is odd]
As done in the even $N$ case, the irreducible class $\mathfrak J_{N,K}$ that contains $\Gamma_{N,K}$ is exactly the unique irreducible component of $\Sigma_{N,K}$ that is also active. This finishes the proof of part (1) of Theorem \ref{thm1}.
\end{proof}

\subsection{\label{sec3.2}At or Below the Structural Threshold: Multiple Recurrent Classes}

In Section \ref{sec3.2}, we prove part (2) of Theorem \ref{thm1}.
Let us assume that 
\begin{equation}
K\le\begin{cases}
\rho^{{\rm st}}_{\star}N^{d} & \text{if}\quad N\enspace\text{is even},\\
N^{d}-\left(\frac{N+1}{2}\right)^{d} & \text{if}\quad N\enspace\text{is odd}.
\end{cases}\label{eq:K-bound}
\end{equation}
Recall that we identify $\mathbb T_N^d$ with $\{0,1,\dots,N-1\}^{d}$. Let us define 
\[
\mathbb{A}^{d}_{N}:=\left\{ \bm{x}\in\mathbb{T}^{d}_{N}:x_{1},\dots,x_{d}\enspace\text{are all even}\right\} ,
\]
and let $\mathbb{B}^{d}_{N}:=\mathbb{T}^{d}_{N}\setminus\mathbb{A}^{d}_{N}$. It is clear that $|\mathbb{A}^{d}_{N}|=\lfloor\frac{N+1}{2}\rfloor^{d}$, thus 
\begin{equation}
|\mathbb{B}^{d}_{N}|=N^{d}-\left\lfloor \frac{N+1}{2}\right\rfloor ^{d}=\begin{cases}
\rho^{{\rm st}}_{\star}N^{d} & \text{if}\quad N\enspace\text{is even},\\
N^{d}-\left(\frac{N+1}{2}\right)^{d} & \text{if}\quad N\enspace\text{is odd}.
\end{cases}\label{eq:ONd-size}
\end{equation}
Let us write 
\[
\Omega_{N,K}:=\left\{ \eta\in\Sigma_{N,K}:\eta(\bm{x})=0\quad\text{for all}\enspace\bm{x}\in\mathbb{A}^{d}_{N}\right\} .
\]
In words, $\Omega_{N,K}$ collects the configurations supported on $\mathbb{B}^{d}_{N}$. The set $\Omega_{N,K}$ is nonempty by \eqref{eq:ONd-size} and the assumption \eqref{eq:K-bound}.

\begin{claim}
\label{claim:closed}The set $\Omega_{N,K}$ is closed for the Markov chain $\eta^{N}_{t}$.
\end{claim}

\begin{proof}
Suppose the contrary, so that there exists a configuration $\eta\in\Omega_{N,K}$ and a legal jump from $\eta$ that enters $\Omega^{c}_{N,K}$. This means that there exist $\bm{x}\in\mathbb{T}^{d}_{N}$ and $\bm{v}\in\mathcal{V}$, with $\eta(\bm{x}-\bm{v})=\eta(\bm{x})=1$ and $\eta(\bm{x}+\bm{v})=0$, such that $\bm{x}-\bm{v},\bm{x}\in\mathbb{B}^{d}_{N}$ and $\bm{x}+\bm{v}\in\mathbb{A}^{d}_{N}$. Let us write $\bm{y}=\bm{x}+\bm{v}\in\mathbb{T}^{d}_{N}$, so that 
\[
\bm{y}-2\bm{v},\bm{y}-\bm{v}\in\mathbb{B}^{d}_{N}\qquad\text{and}\qquad\bm{y}\in\mathbb{A}^{d}_{N}.
\]
If $N$ is even, then $\bm{y}\in\mathbb{A}^{d}_{N}$ automatically implies that $\bm{y}-2\bm{v}\in\mathbb{A}^{d}_{N}$, which gives a contradiction. If $N$ is odd, then without loss of generality assume that $\bm{v}\in\{\pm\bm{e}_{1}\}$. Then, the two facts $\bm{y}\in\mathbb{A}^{d}_{N}$ and $\bm{y}-\bm{v}\in\mathbb{B}^{d}_{N}$ imply that the two cases $y_{1}=0$ \& $\bm{v}=+\bm{e}_{1}$ and $y_{1}=N-1$ \& $\bm{v}=-\bm{e}_{1}$ are forbidden, i.e., the jump $\bm{y}\to\bm{y}-\bm{v}$ cannot pass through the boundary of the torus. This gives $\bm{y}-2\bm{v}\in\mathbb{A}^{d}_{N}$, which yields a contradiction. Thus, our claim holds.
\end{proof}

\begin{proof}
[Proof of Theorem \ref{thm1}-(2)]
By Claim \ref{claim:closed}, the set $\Omega_{N,K}$ is nonempty and closed, thus it contains an irreducible component $\mathfrak{I}_{N,K}$. Finally, define for each $\bm a \in \{0,1\}^d$ as
\[
\Omega_{N,K}^{\bm a}:=\left\{ \eta\in\Sigma_{N,K}:\eta(\bm{x})=0\quad\text{for all}\enspace\bm{x}\in\mathbb{A}^{d}_{N}+\bm a\right\} ,
\]
where $A+\bm a:=\{\bm{x}+\bm a:\bm{x}\in A\}$. Since the torus is translation invariant, the same logic yields that $\Omega_{N,K}^{\bm a}$ is also nonempty and closed for all $\bm a \in \{0,1\}^d$, thus contains an irreducible component $\mathfrak{I}_{N,K}^{\bm a}$. If, on the contrary, $\Sigma_{N,K}$ has a unique irreducible component, then we should have $\mathfrak I_{N,K}^{\bm a} = \mathfrak I_{N,K}$ for all $\bm a \in \{0,1\}^d$. In particular, we should have
\[
\mathfrak I_{N,K} \subset \bigcap_{\bm a \in \{0,1\}^d} \Omega_{N,K}^{\bm a}.
\]
However, the set on the right-hand side is clearly empty, since the sets $\{ \mathbb A_N^d + \bm a : \bm a \in \{0,1\}^d \}$ cover the whole periodic lattice $\mathbb T_N^d$. This contradicts the assumption and thus concludes the proof.
\end{proof}

\section{\label{sec4}Island Decomposition and Proof of Theorem \ref{thm2}}
In this section, we prove Theorem \ref{thm2}. Throughout, $d \ge 2$ is fixed and $N$ is sufficiently large.

\subsection{Island Decomposition}
Recall that we write $\eta \rightsquigarrow \xi$ if a finite sequence of legal jumps leads from $\eta$ to $\xi$, allowing a sequence of length zero. Let $\mathcal R (\eta) := \{ \xi : \eta \rightsquigarrow \xi \}$. Write $\eta\rightarrowtail\xi$ if $\xi$ is obtained from $\eta$ by a single legal jump, or in other words, if $\xi=\eta^{\bm{x},\bm{x}+\bm{v}}$ for some $\bm{x}\in\mathbb{T}^{d}_{N}$ and $\bm{v}\in\mathcal{V}$ (cf. \eqref{eq:V-def}) where $\eta(\bm x - \bm v) = \eta(\bm x) =1$ and $\eta ( \bm x + \bm v) = 0$.

\begin{defn}[Islands]\label{def:island}
A nonempty set $I \subset \mathbb{T}^{d}_{N}$ is an \emph{island} of $\eta \in \Sigma_N$ if for every $\xi \in \mathcal R(\eta)$ and every nearest-neighbor pair $\bm x \in I$, $\bm y \in I^c$, the following hold: $\xi(\bm x) \xi(\bm y) = 0$, and no legal jump from $\xi$ crosses the edge $\{\bm x , \bm y\}$ in either direction.
An island $I$ is \emph{minimal} if no proper subset of $I$ is an island. Denote by $\mathscr{I}(\eta)$ the collection of all minimal islands.
\end{defn}

This definition is equivalent to the almost-sure dynamical definition: every finite legal path has positive probability of being followed by the jump chain, and the state space is finite. The entire torus is an island, and the complement of a proper island is an island.

For completeness, form a graph on $\mathbb T_N^d$ by retaining a nearest-neighbor edge whenever its endpoints can be simultaneously occupied in a reachable configuration, or whenever a legal jump across it is possible from a reachable configuration. The islands are precisely the nonempty unions of components of this graph. Consequently, the minimal islands partition $\mathbb T_N^d$ and are connected in the nearest-neighbor graph. In particular, the islands are determined by the initial configuration and legal reachability, independently of a realization of the Poisson clocks.

An island initially containing no particles stays empty, because no jump can cross its boundary. Each of its singletons is then itself an island. Thus an initially empty minimal island is a singleton.

A partition is called $\gamma$-nice if each of its parts has cardinality at most $\gamma$.

\begin{thm}
\label{thm:key}There exists a constant $\rho_{0}=\rho_{0}(d)\in(0,2^{-d})$ such that, for all $\rho\in(0,\rho_{0})$, there is a constant $c_0 = c_0 (\rho,d) > 0$ satisfying
\[
\lim_{N\to\infty}\nu^{N}_{\rho}\left(\eta\in\Sigma_{N}:\mathscr{I}(\eta)\quad\text{is}\enspace c_{0}(\log N)^{d}\text{-nice}\right)=1.
\]
\end{thm}

\begin{defn}
[Dense and thick sets]\label{def:dense-thick}Given $\eta\in\Sigma_{N}$, a nonempty connected set $C\subset\mathbb{T}^{d}_{N}$ is called \emph{dense} with respect to $\eta$ if the particle density of $\eta$ in $C$ is at least $2^{-d}$, i.e., if  
\[
\sum_{\bm{x}\in C}\eta(\bm{x})\geq\frac{|C|}{2^{d}}.
\]
A connected set $C\subset\mathbb{T}^{d}_{N}$ is \emph{thick} with respect to $\eta$ if $C$ is a finite union of dense sets.

\end{defn}

\begin{lem}
\label{l1}For $\eta\in\Sigma_{N}$, every minimal island of $\eta$ with at least one particle is thick.
\end{lem}

\begin{proof}
Fix such a minimal island $I$ and let $P=\{\bm{x}\in I:\eta(\bm{x})=1\}\ne\emptyset$. Let $\Xi(\eta)$ be the collection of all finite legal paths $\omega = (\omega_0, \dots ,\omega_\ell)$ with $\omega_0 = \eta$, including $\ell = 0$. For such a path put
\[
P_j (\omega) = \left\{ \bm x \in I : \omega_j (\bm x ) = 1 \right\}, \qquad M(\omega) = \bigcup_{j=0}^\ell P_j (\omega), \qquad M = \bigcup_{\omega \in \Xi(\eta)} M(\omega).
\]
We first show that
\begin{equation}
M=I.\label{eq:claim1}
\end{equation}
Clearly $\emptyset \ne P \subset M \subset I$. Every site of $I \setminus M$ is vacant in every reachable configuration; a legal jump into such a site would contradict the definition of $M$. Hence no simultaneous occupation or legal jump can occur across a boundary edge between $M$ and $I \setminus M$. Boundary edges between $M$ and $I^c$ are covered by the island property of $I$. Thus $M$ is an island, and minimality gives $M=I$.

Fix a finite path $\omega$ and a connected component $C$ of $M(\omega)$. For $j \in \llbracket 0,\ell \rrbracket$, define
\[
P_j^C = P_j(\omega) \cap C.
\]
Both endpoints of each particle jump inside $I$ belong to the same connected component of $M(\omega)$. No particle crosses the boundary of $I$. It follows that
\[
|P_j^C| = |P \cap C| = \sum_{\bm z \in C} \eta(\bm z), \qquad j \in \llbracket 0,\ell \rrbracket.
\]

We claim that any elementary box intersecting $C$ contains a particle from $C$ at the end of the path:
\[
\mathcal B_{\bm x} \cap C \ne \emptyset \qquad \Longrightarrow \qquad \mathcal B_{\bm x} \cap P_\ell^C \ne \emptyset.
\]
Indeed, such a box is visited by a particle from $C$ at some step. If its last particle from $C$ were subsequently to leave, say by $\bm z \to \bm z + \bm v$, then the facilitating site $\bm z - \bm v$ would still lie in $\mathcal B_{\bm x}$. The facilitator belongs to $I$, since it is occupied and adjacent to the occupied site $\bm z \in I$. It also belongs to $M(\omega)$ and is adjacent to $\bm z \in C$, so it belongs to $C$. This particle remains in the box after the jump, a contradiction. This proves the claim.

Let $A_C = \{ \bm x \in \mathbb T_N^d : \mathcal B_{\bm x} \cap C \ne \emptyset \}$. Every $\bm z \in C$ is an anchor in $A_C$, and every site belongs to exactly $2^d$ boxes indexed by their anchors. Therefore
\[
|C| \le |A_C| \le \sum_{\bm x \in A_C} |\mathcal B_{\bm x} \cap P_\ell^C | \le \sum_{\bm z \in P_\ell^C} \sum_{\bm x \in \mathbb T_N^d} {\bf 1}_{\{ \bm z \in \mathcal B_{\bm x} \}} = 2^d |P_\ell^C| = 2^d \sum_{\bm z \in C} \eta(\bm z).
\]
Thus $C$ is dense. By \eqref{eq:claim1}, $I$ is a union of these dense connected components, taken over all finite legal paths. Since $I$ is finite, a finite subcollection covers it. Hence $I$ is thick.
\end{proof}

\begin{lem}
\label{l2}Suppose that $k\ge1$ and that $E_{1},E_{2},\dots,E_{k}\subset\mathbb{T}^{d}_{N}$ are nonempty connected sets such that $E:=\bigcup^{k}_{i=1}E_{i}$ is also connected. Then, there exists a subset $S\subset\{1,2,\dots,k\}$ satisfying the following three conditions:

\begin{enumerate}
\item The union $\bigcup_{i\in S}E_{i}$ is a connected set.

\item For any distinct $i_{1},i_{2},i_{3}\in S$, we have $E_{i_{1}}\cap E_{i_{2}}\cap E_{i_{3}}=\emptyset$.

\item We have $|\bigcup_{i\in S}E_{i}|\ge\frac{1}{2}|E|^{1/d}$.

\end{enumerate}

\end{lem}

\begin{proof}
Denote by $J_{p}$ the projection of $E$ onto the $p$-th axis: $J_{p}:=\Phi_{p}(E)$ where $\Phi_{p}:\mathbb{T}^{d}_{N}\to\mathbb{T}_{N}$ is defined as $\Phi_{p}(\bm{x}):=x_{p}$. Since $E$ is connected, its projection $J_{p}$ is also connected. We divide into two cases.

\begin{itemize}
\item \textbf{(Case 1)} First, suppose that $J_{p}\subsetneq\mathbb{T}_{N}$ for all $p\in\llbracket1,d\rrbracket$. In this case, clearly $E\subset\prod^{d}_{p=1}J_{p}$ thus 
\[
|E|\le\prod^{d}_{p=1}|J_{p}|.
\]
Without loss of generality, we may assume that $|J_{1}|\ge|E|^{1/d}$. Let $\mathfrak{r}\in E$ be any point in $E$ such that $\Phi_{1}(\mathfrak{r})$ is the rightmost point in $J_{1}$ and let $\mathfrak{l}\in E$ be any point such that $\Phi_{1}(\mathfrak{l})$ is the leftmost point in $J_{1}$ (such that $J_{1}$ equals the interval from $\Phi_{1}(\mathfrak{l})$ to $\Phi_{1}(\mathfrak{r})$ in the increasing order). Define $\mathcal{G}$ as the collection of all possible index sequences $(j_{0},j_{1},\dots j_{m})$ which satisfy: $\mathfrak{l}\in E_{j_{0}}$, $\mathfrak{r}\in E_{j_{m}}$, and $E_{j_{q}}\cup E_{j_{q+1}}$ is connected for all $q\in\llbracket0,m-1\rrbracket$. First, we prove that $\mathcal{G}$ is nonempty. Since $E$ is a connected set, there exists a path $\mathfrak{l}=\bm{x}_{0},\dots,\bm{x}_{t}=\mathfrak{r}$ inside $E$, i.e., $\bm{x}_{n}\in E$ and $\bm{x}_{n}\sim\bm{x}_{n+1}$ for all $n$. Taking any $j_{n}\in\llbracket1,k\rrbracket$ such that $\bm{x}_{n}\in E_{j_{n}}$ for each $n$ implies that $(j_{0},j_{1},\dots,j_{t})\in\mathcal{G}$, thus $\mathcal{G}$ is nonempty. Choose an element $(s_{0},s_{1},\dots,s_{m})\in\mathcal{G}$ with the minimal length $m$. We now show that the index set $S:=\{s_{0},s_{1},\dots,s_{m}\}$ satisfies the required conditions. First, by construction, $\bigcup^{m}_{i=0}E_{s_{i}}$ is connected. Second, by the minimality of $m$, it must hold that $E_{s_{q}}\cap E_{s_{r}}=\emptyset$ for any $0\le q<r\le m$ with $r-q>1$; otherwise, we would have $E_{s_{q}}\cap E_{s_{r}}\ne\emptyset$ thus $E_{s_{q}}\cup E_{s_{r}}$ is connected, and we could find a shorter sequence $(s_{0},\dots,s_{q},s_{r},\dots,s_{m})\in\mathcal{G}$, contradicting the minimality. This ensures that no three distinct indices $s,s',s''\in S$ satisfy $E_{s}\cap E_{s'}\cap E_{s''}\ne\emptyset$, proving the second condition. Lastly, since $\Phi_{1}(\bigcup_{i\in S}E_{i})$ is a connected subset of $J_{1}$ which contains both $\Phi_1 (\mathfrak{l})$ and $\Phi_1 (\mathfrak{r})$ we have $\Phi_{1}(\bigcup_{i\in S}E_{i})=J_{1}$, thus we obtain 
\[
\frac{1}{2}|E|^{1/d}\le|E|^{1/d}\leq|J_{1}|\le\left|\bigcup_{i\in S}E_{i}\right|.
\]

\item \textbf{(Case 2)} Next, if there exists $p\in\llbracket1,d\rrbracket$ such that $J_{p}=\mathbb{T}_{N}$, then without loss of generality suppose that $J_{1}=\mathbb{T}_{N}$. Clearly, $|J_{1}|=N=|\mathbb{T}^{d}_{N}|^{1/d}\ge|E|^{1/d}$. Then, we may set $\mathfrak l,\mathfrak r \in E$ such that $\Phi_1(\mathfrak{l})=0$, $\Phi_1(\mathfrak{r})=\lfloor\frac{N-1}{2}\rfloor$, and apply the exact same idea as above to find a set of indices $S$ which satisfies the desired three conditions. The only difference exists at the last step, where a connected subset of $\mathbb{T}_{N}$ which contains both $0$ and $\lfloor\frac{N-1}{2}\rfloor$ now has size at least $\frac{N}{2}$, thus 
\[
\frac{1}{2}|E|^{1/d}\leq\frac{N}{2}\le\left|\bigcup_{i\in S}E_{i}\right|,
\]
as desired.
\end{itemize}
\end{proof}

Set
\begin{equation}
\gamma := 2^{-d-1}, \qquad f(d):= \gamma \exp\left(-\frac{1+\log2+\log(2d-1)}{\gamma}\right).\label{eq:fd-def}
\end{equation}
For each $\eta\in\Sigma_{N}$ and nonempty $F\subset\mathbb{T}^{d}_{N}$, denote by ${\bf d}_{F}(\eta)$ the local density of $\eta$ in $F$: 
\[
{\bf d}_{F}(\eta):=\frac{\sum_{\bm{x}\in F}\eta(\bm{x})}{|F|}.
\]

\begin{lem}
\label{l3}If $0<\rho<f(d)$, there exists $c_{1}=c_{1}(\rho,d)>0$ such that 
\begin{equation}
\lim_{N\to\infty}\nu^{N}_{\rho}\left(\eta\in\Sigma_{N}:\exists\text{ connected}\enspace F\subset\mathbb{T}^{d}_{N}\quad\text{such that}\enspace|F|\ge c_{1}\log N,\enspace{\bf d}_{F}(\eta)\ge\gamma\right)=0.\label{eq:l3}
\end{equation}
\end{lem}

\begin{proof}
By a standard connected-set counting bound (cf. \cite[Problem 45]{Bol06}), the number of connected vertex sets of $\mathbb{T}^{d}_{N}$ with $u>1$ vertices is bounded by $N^{d}((2d-1)e)^{u-1}$. By applying the union bound over all such subsets and the large deviation principle for Bernoulli random variables, we obtain
\begin{align*}
\nu^{N}_{\rho} & \left(\eta\in\Sigma_{N}:\exists\text{ connected}\enspace F\subset\mathbb{T}^{d}_{N}\quad\text{such that}\enspace|F|=u,\enspace{\bf d}_{F}(\eta)\ge\gamma\right)\\
 & \le N^{d}((2d-1)e)^{u-1}e^{-uh_{\rho}(\gamma)}\le N^{d}b(\rho,d)^{-u},
\end{align*}
where $h_{\rho}(\gamma)=\gamma\log\frac{\gamma}{\rho}+(1-\gamma)\log\frac{1-\gamma}{1-\rho}$ is the large deviations rate function and $b(\rho,d):=\frac{e^{h_{\rho}(\gamma)}}{(2d-1)e}$. Since $0<\rho<f(d)<\gamma$ (cf. \eqref{eq:fd-def}), we have: 
\begin{align*}
 & \log b(\rho,d)\ge h_{f(d)}(\gamma)-1-\log(2d-1)\\
 & >\gamma\log\frac{\gamma}{f(d)}+(1-\gamma)\log(1-\gamma)-1-\log(2d-1)=\log2+(1-\gamma)\log(1-\gamma)>0.
\end{align*}
In the second inequality, we used the fact that $\log\frac{\gamma}{f(d)}=2^{d+1}(1+\log2+\log(2d-1))$. Therefore, taking the sum over all possible sizes $u\ge c_{1}\log N$, we may bound the probability in the left-hand side of \eqref{eq:l3} above with 
\[
\sum_{u\ge c_{1}\log N}N^{d}b(\rho,d)^{-u}\le N^{d}\frac{b(\rho,d)^{1-c_{1}\log N}}{b(\rho,d)-1}\le\frac{b(\rho,d)}{b(\rho,d)-1}\frac{1}{N^{c_{1}\log b(\rho,d)-d}}.
\]
So, by taking $c_{1}=c_{1}(\rho,d)=\frac{2d}{\log b(\rho,d)}>\frac{d}{\log b(\rho,d)}$, the right-hand side tends to $0$ as $N\to\infty$.
\end{proof}

\begin{proof}
[Proof of Theorem \ref{thm:key}] Define $\rho_{0}:=f(d)$, take $\rho\in(0,\rho_{0})$, and let $c_{0}(\rho,d):=(2c_{1})^{d}$, where $f(d)$ and $c_{1}$ are the constants from Lemma \ref{l3}, and take $\eta\sim\nu^{N}_{\rho}$. If $\mathscr{I}(\eta)$ is not $c_{0}(\log N)^{d}$-nice, then there exists $I\in\mathscr{I}(\eta)$ such that $|I|>c_{0}(\log N)^{d}$.
Since an initially empty minimal island is a singleton and $|I|>c_0(\log N)^d>1$ for sufficiently large $N$, the island $I$ contains at least one particle. Hence, by Lemma \ref{l1}, $I$ is thick. By the definition of thickness and Lemma \ref{l2}, there exists a connected set $\widetilde{I}\subset I$ such that $\widetilde{I}=H_{1}\cup\cdots\cup H_{r}$ where $H_{1},H_{2},\dots,H_{r}$ are dense, 
\[
H_{i_{1}}\cap H_{i_{2}}\cap H_{i_{3}}=\emptyset\qquad\text{for any distinct}\quad i_{1},i_{2},i_{3},
\]
and $|\widetilde{I}|\geq\frac{1}{2}|I|^{1/d}>c_{1}\log N$. Since each $H_{i}$ is dense and any particle in $\widetilde{I}$ belongs to at most two among $H_{1},\dots,H_{r}$, we obtain that 
\[
{\bf d}_{\widetilde{I}}(\eta)=\frac{\sum_{\bm{x}\in\widetilde{I}}\eta(\bm{x})}{|\widetilde{I}|}\ge\frac{\sum^{r}_{i=1}\sum_{\bm{x}\in H_{i}}\eta(\bm{x})}{2|\widetilde{I}|}=\frac{\sum^{r}_{i=1}|H_{i}|{\bf d}_{H_{i}}(\eta)}{2|\widetilde{I}|}\ge \gamma.
\]
Thus, Lemma \ref{l3} completes the proof.
\end{proof}

The event given in Lemma \ref{l3} is increasing in the initial occupation variables. Thus, we also have the absorption and upper-bound extension under stochastic domination
stated in Remark \ref{rem:thm2}.

\subsection{Absorbing Property at Low Density}
By Theorem \ref{thm:key}, with high probability, we may fix an initial configuration $\eta\sim\nu^{N}_{\rho}$ such that $\mathscr{I}(\eta)$ is $c_{0}(\log N)^{d}$-nice. Under $\mathbb P_\eta^N$, the processes restricted to the islands evolve independently. Write 
\[
\mathscr{I}(\eta)=\{I_{1},\dots,I_{m}\},
\]
such that $|I_{i}|\le c_{0}(\log N)^{d}$ for $i\in\llbracket1,m\rrbracket$, where $m=|\mathscr{I}(\eta)|$. Then, by definition, we may interpret the global FEP system $\eta^{N}_{t}$ as $m$ \emph{independent} FEP systems  
\begin{equation}
(\eta^{N,1}_{t},\dots,\eta^{N,m}_{t}),\label{eq:low-density-dec}
\end{equation}
where each process $\eta^{N,i}_{t}$ is defined on the subset $I_{i}$. Indeed, the $m$ processes are independent since by Definition \ref{def:island}, any two particles in different islands cannot facilitate each other and particles cannot jump between different islands. In this sense, we have 
\begin{equation}
\tau_{{\rm tr}}=\tau^{1}_{{\rm tr}}\vee\cdots\vee\tau^{m}_{{\rm tr}},\label{eq:t-rec-dec}
\end{equation}
where $\tau^{i}_{{\rm tr}}$ is the transience time of the $i$-th system $\eta^{N,i}_{t}$ on $I_{i}$. Let us write $K_{i}:=\sum_{\bm{x}\in I_{i}}\eta(\bm{x})$, the number of particles on $I_{i}$.

Now, fix $i\in\llbracket1,m\rrbracket$. Since $|I_{i}|\le c_{0}(\log N)^{d}<N$, if we project $I_i$ onto each axis there exists a missing coordinate. This means that we can embed $I_i$ into the lattice $\mathbb{Z}^{d}$, and consider $\eta^{N,i}_{t}$ as a particle system on $\{0,1\}^{\mathbb{Z}^{d}}$. Notice that from any particle configuration on $I_{i}$ reachable from $\eta$, one may simply push all particles in the $+\bm{e}_{k}$, $k\in\llbracket1,d\rrbracket$ directions as much as possible, and thereby obtain an absorbing state within finite steps. This implies that all recurrent states are absorbing states, thus at time $\tau^{i}_{{\rm tr}}$ the process $\eta^{N,i}_{t}$ arrives at an absorbing state. Applying this fact to each $i\in\llbracket1,m\rrbracket$, we obtain that 
\begin{equation}
\lim_{N\to\infty}\mathbb{P}^{N}_{\nu_{\rho}}\left[\eta^{N}_{\tau_{{\rm tr}}}\in\Sigma^{{\rm abs}}_{N}\right]=1.\label{eq:absorbing}
\end{equation}

\subsection{Upper Bound of Theorem \ref{thm2}}
In this subsection, we prove the upper bound part of Theorem \ref{thm2}. We continue to interpret as in \eqref{eq:low-density-dec} and consider each system $\eta^{N,i}_{t}$ on $\{0,1\}^{\mathbb{Z}^{d}}$ with $K_{i}$ particles.

\begin{lem}
\label{l4}For each $i\in\llbracket1,m\rrbracket$, 
\[
\mathbb{E}^{N}_{\eta}\left[\tau^{i}_{{\rm tr}}\right]\le\frac{|I_{i}|^{3}}{4},
\]
and, for every $t\ge0$,  
\begin{equation}
\mathbb{P}^{N}_{\eta}\left(\tau^{i}_{{\rm tr}}\ge t\right)\le2\exp\left\{ -\frac{(2\log2)t}{|I_{i}|^{3}}\right\} .\label{eq:island-tail-revised}
\end{equation}
\end{lem}

\begin{proof}
Fix $\bm{a}=(a_{1},\dots,a_{d})\in I_{i}$ and define the spatial second moment  
\[
Q(\xi):=\sum_{\bm{x}\in I_{i}}\xi(\bm{x})\|\bm{x}-\bm{a}\|^{2}_{2}\qquad\text{for}\quad\xi\in\{0,1\}^{\mathbb{Z}^{d}}.
\]
Every $\bm{x}\in I_{i}$ can be joined to $\bm{a}$ by a path in $I_{i}$ of at most $|I_{i}|-1$ edges. Hence, for every $t\ge0$, recalling that $K_{i}$ denotes the number of particles, 
\begin{equation}
0\le Q(\eta^{N,i}_{t})=\sum_{\bm{x}\in I_{i}}\eta^{N,i}_{t}(\bm{x})\|\bm{x}-\bm{a}\|^{2}_{2}\le K_{i}(|I_{i}|-1)^{2}\le|I_{i}|^{3}.\label{eq:coordinate-bound-revised}
\end{equation}
Fix a coordinate direction $k\in\llbracket1,d\rrbracket$ and consider a maximal occupied run of length $\ell\ge2$ along a line parallel to $\bm{e}_{k}$. Its two outward endpoint jumps each have rate one, and both stay in $I_{i}$. If the endpoint coordinates relative to $a_{k}$ are $u$ and $v$, their combined contribution to $\mathcal{L}_{N}Q$ is  
\[
\left[(u-1)^{2}-u^{2}\right]+\left[(v+1)^{2}-v^{2}\right]=2(v-u+1)=2\ell.
\]
Denote by $P^{i}_{t}$ the set of positions of the particles at time $t$, such that $|P^{i}_{t}|=K_{i}$. Summing over all runs and all coordinate directions gives  
\begin{equation}
(\mathcal{L}_{N}Q)(\eta^{N,i}_{t})=2\sum^{d}_{k=1}\sum_{\bm{x}\in P^{i}_{t}}{\bf 1}\left\{ \bm{x}-\bm{e}_{k}\in P^{i}_{t}\quad\text{or}\quad\bm{x}+\bm{e}_{k}\in P^{i}_{t}\right\} \ge4,\qquad t<\tau^{i}_{{\rm tr}}.\label{eq:quadratic-drift-revised}
\end{equation}
The inequality holds because a non-absorbing configuration contains an occupied run of length at least two. Applying Dynkin's formula at $t\wedge\tau^{i}_{{\rm tr}}$ and using \eqref{eq:coordinate-bound-revised}, we obtain  
\begin{equation}
4\mathbb{E}^{N}_{\eta}\left[t\wedge\tau^{i}_{{\rm tr}}\right]\le\mathbb{E}^{N}_{\eta}\left[Q(\eta^{N,i}_{t\wedge\tau^{i}_{{\rm tr}}})\right]-Q(\eta)\le|I_{i}|^{3}.\label{eq:island-mean-revised}
\end{equation}
Letting $t\to\infty$ gives the asserted bound on the expectation. Mind that this estimate is uniform over all configurations reachable from $\eta$.

Let $B:=|I_{i}|^{3}/2$ and let $(\mathcal{F}_{t})_{t\ge0}$ be the natural filtration. The Markov property and Markov's inequality therefore imply 
\[
\mathbb{P}^{N}_{\eta}\left(\tau^{i}_{{\rm tr}}\ge(n+1)B\mid\mathcal{F}_{nB}\right)\le\frac{1}{2}{\bf 1}\left\{ \tau^{i}_{{\rm tr}}\ge nB\right\} ,\qquad n\ge0.
\]
Taking expectations and iterating, we obtain 
\[
\mathbb{P}^{N}_{\eta}\left(\tau^{i}_{{\rm tr}}\ge nB\right)\le2^{-n}.
\]
Taking $n=\lfloor t/B\rfloor$ proves \eqref{eq:island-tail-revised}.
\end{proof}

\begin{lem}
\label{lem:thm2-UB}For $\rho_{0}$ as chosen in Theorem \ref{thm:key} and $\rho\in(0,\rho_{0})$, 
\[
\lim_{N\to\infty}\mathbb{P}^{N}_{\nu_{\rho}}\left[\tau_{{\rm tr}}<c(\log N)^{3d+1}\right]=1.
\]
\end{lem}

\begin{proof}
Fix $0<\rho<\rho_{0}$, with $\rho_{0}$ as in Theorem \ref{thm:key}, and write $L_{N}:=c_{0}(\log N)^{d}$. Suppose that $\mathscr{I}(\eta)$ is $L_{N}$-nice, which happens with high probability if $\eta\sim\nu^{N}_{\rho}$ by Theorem \ref{thm:key}. Since $|\mathscr{I}(\eta)|\le N^{d}$, the decomposition \eqref{eq:t-rec-dec}, Lemma \ref{l4}, and the union bound imply  
\begin{equation}
\mathbb{P}^{N}_{\eta}\left(\tau_{{\rm tr}}\ge T\right)\le\sum^{|\mathscr{I}(\eta)|}_{i=1}\mathbb{P}^{N}_{\eta}\left(\tau^{i}_{{\rm tr}}\ge T\right)\le2N^{d}\exp\left\{ -\frac{2(\log2)T}{L^{3}_{N}}\right\} \qquad\text{for any}\quad T\ge0.\label{eq:global-tail-revised}
\end{equation}
Choose a constant $c=(d+1)c^{3}_{0}/(2\log2)$ and set $T=c(\log N)^{3d+1}$, then the right-hand side of \eqref{eq:global-tail-revised} is at most $2N^{-1}$ and thus vanishes as $N\to\infty$. Therefore, Theorem \ref{thm:key} gives 
\[
\mathbb{P}^{N}_{\nu^{N}_{\rho}}\left(\tau_{{\rm tr}}<c(\log N)^{3d+1}\right)\to1\qquad\text{as}\quad N\to\infty.
\]
\end{proof}

\subsection{Lower Bound of Theorem \ref{thm2}}
Finally, we prove the lower bound part and conclude the proof of Theorem \ref{thm2}, i.e., we prove that for some constant $c>0$, 
\begin{equation}
\lim_{N\to\infty}\mathbb{P}^{N}_{\nu_{\rho}}\left[\tau_{{\rm tr}}>c\log N\right]=1.\label{eq:thm2-LB}
\end{equation}
Fix $\rho\in(0,\rho_{0})$. Notice that by a standard large deviations principle, there are at least $b_{N}\ge c_2 N^{d}$ pairwise disjoint fully occupied elementary boxes (cf. \eqref{eq:box-def}) with high probability, as long as $c_2$ is small enough. Let us sample an initial configuration $\eta \sim \nu_\rho^N$ with the desired property. Denote by $H_{1},\dots,H_{b_{N}}$ the first ringings of independent Poisson clocks on the collection of all directed edges inside each of the $b_{N}$ distinct boxes, which represent the facilitating mechanism of the FEP system. Notice that each clock has rate $d2^d$ since there are exactly $d2^d$ directed edges in each box. Moreover, before each clock rings, the elementary box remains fully occupied and thus the resulting configuration cannot be an absorbing state. Therefore, by \eqref{eq:absorbing}, 
\begin{align*}
\mathbb{P}^{N}_\eta \left[\tau_{{\rm tr}}\le c\log N\right] & \le \mathbb{P}^{N}_\eta \left[H_{1}\vee\cdots\vee H_{b_{N}}\le c\log N\right]+o(1)\\
 & =\prod^{b_{N}}_{j=1}\mathbb{P}^{N}_\eta \left[H_{j}\le c\log N\right]+o(1)\le\left(1-N^{-cd2^d}\right)^{c_2 N^{d}}+o(1).
\end{align*}
Therefore, it suffices to choose $c<2^{-d}$ to verify \eqref{eq:thm2-LB} via 
\[
\left(1-N^{-cd2^d}\right)^{c_2 N^{d}}\le e^{-N^{-cd2^d}c_2 N^{d}}=e^{-c_2 N^{d(1-c2^d)}}\to0\qquad\text{as}\quad N\to\infty.
\]

\begin{proof}
[Proof of Theorem \ref{thm2}] It remains to collect \eqref{eq:absorbing}, Lemma \ref{lem:thm2-UB}, and \eqref{eq:thm2-LB}.
\end{proof}

\section{\label{sec5}Proof of Theorem \ref{thm3}}
In this section, we prove the following quantitative version of Theorem \ref{thm3}.

\begin{thm}
\label{thm:plb-main}Fix an integer $d\ge2$ and $\rho\in(1/2,3/4)$, and set  
\begin{equation}
\gamma_{\rho}=\frac{1}{2}\min\left\{ \rho-\frac{1}{2},\frac{3}{4}-\rho\right\} ,\qquad c_{\rho,d}=\frac{\gamma^{2}_{\rho}}{(384e)d(d+2)}.\label{eq:plb-constants}
\end{equation}
Then for all sufficiently large even $N$,  
\begin{equation}
\mathbb{P}^{N}_{\nu_{\rho}}\left(\tau^{N}_{{\rm tr}}\le c_{\rho,d}\frac{N}{\log N}\right)\le N^{-4}+2N^{-3}.\label{eq:plb-main-bound}
\end{equation}
\end{thm}

\begin{proof}
[Proof of Theorem \ref{thm3}] It follows directly from Theorem \ref{thm:plb-main}.
\end{proof}

In the remaining part of the section, we prove Theorem \ref{thm:plb-main}. Hereafter, let us assume that $N$ is sufficiently large and even. Recall that we write $\eta\rightsquigarrow\zeta$ if a finite sequence of legal jumps leads from $\eta$ to $\zeta$, allowing a sequence of length zero. Write $\bm{x}=(x_{1},x_{2},\bm{w})\in\mathbb{T}^{d}_{N}$, where $\bm{w}\in\mathbb{T}^{d-2}_{N}$, and put  
\[
\mathbb{H}=\left\{ \bm{x}\in\mathbb{T}^{d}_{N}:x_{1},x_{2}\enspace\text{are even}\right\} ,\qquad\mathbb{S}=\mathbb{T}^{d}_{N}\setminus\mathbb{H}.
\]
Then $|\mathbb{H}|=N^{d}/4$. Define  
\begin{equation}
\Lambda_{N}=\left\{ \eta\in\Sigma_{N}:\eta(\bm{z})=0\quad\text{for every}\enspace\bm{z}\in\mathbb{H}\right\} .\label{eq:plb-closed-set}
\end{equation}

\begin{lem}
\label{lem:plb-closed-set}The set $\Lambda_{N}$ is closed under the full $d$-dimensional dynamics.

\end{lem}

\begin{proof}
Fix $\xi \in \Lambda_N$ and suppose that a legal jump $\bm x \to \bm x + \bm v$ enters $\mathbb H$. Since $N$ is even,
$\bm x - \bm v = (\bm x + \bm v) - 2 \bm v \in \mathbb H$.
Thus the facilitating site is vacant, contradicting legality. Hence no legal jump leaves $\Lambda_N$.
\end{proof}

\subsection{A Deterministic Criterion for Transience}
Let $\mathcal{E}_{N,k}$ be the set of one-dimensional configurations on $\mathbb{T}_{N}$ with $k$ particles and with no two adjacent vacancies, including across the periodic boundary. For $N/2 < k < N$, this set $\mathcal E_{N,k}$ is the unique recurrent class in the $k$-particle sector of the one-dimensional FEP; see \cite[Section 2.2.1, Lemma 2.1, and Section 3]{BESS20}. We record the properties needed below.

\begin{lem}
\label{lem:plb-one-dimensional}If $N/2<k<N$, every one-dimensional configuration on $\mathbb{T}_{N}$ with $k$ particles can reach $\mathcal{E}_{N,k}$ by a finite sequence of legal facilitated jumps, and the set $\mathcal{E}_{N,k}$ is irreducible. In particular, from any configuration in this set one can place a vacancy, or a particle, at any prescribed site using only such jumps.
\end{lem}
For $\eta\in\Sigma_{N}$ and $\bm{y}\in\mathbb{T}^{d-1}_{N}$, define the periodic line $L_{\bm{y}}$ and its particle count by  
\[
L_{\bm{y}}=\left\{ (x,\bm{y}):x\in\mathbb{T}_{N}\right\} ,\qquad k_{\bm{y}}(\eta)=\sum_{x\in\mathbb{T}_{N}}\eta(x,\bm{y}).
\]

\begin{lem}
\label{lem:plb-transient}Suppose that $N$ is even and $\eta$ satisfies  
\begin{equation}
\frac{1}{2}<\frac{k_{\bm{y}}(\eta)}{N}<\frac{3}{4}\qquad\text{for every}\quad\bm{y}\in\mathbb{T}^{d-1}_{N}.\label{eq:plb-row-interval}
\end{equation}
Then $\eta\notin\Sigma^{{\rm rec}}_{N}$.
\end{lem}

\begin{proof}
We construct a finite legal path from $\eta$ into $\Lambda_{N}$. Initially $\eta\notin\Lambda_{N}$: a line with even second coordinate contains more than $N/2$ particles but has only $N/2$ sites with odd first coordinate, so some site of $\mathbb{H}$ must be occupied. Since $\Lambda_{N}$ is a closed subset, such a path will prove the transience of $\eta$.

First apply Lemma \ref{lem:plb-one-dimensional} on every line $L_{\bm{y}}$, using only jumps in directions $\pm\bm{e}_{1}$, so that its configuration belongs to $\mathcal{E}_{N,k_{\bm{y}}(\eta)}$. These moves preserve all line counts. Next process the slices  
\[
\Pi_{\bm{w}}=\left\{ (x,y,\bm{w}):x,y\in\mathbb{T}_{N}\right\} ,\qquad\bm{w}\in\mathbb{T}^{d-2}_{N},
\]
in any fixed order. Within each slice use only directions $\pm\bm{e}_{1},\pm\bm{e}_{2}$. Every jump and its facilitator then lie in that slice, so these are legal jumps of the full process and do not change any other slice. Additional legal jumps in other directions do not prevent us from prescribing this finite path.

Fix one slice $\Pi_{\bm{w}}$ and suppress $\bm{w}$ in its coordinates and line counts for the remainder of the slice construction. Call $x$ the column coordinate and $y$ the row coordinate. Process the row pairs $(2i,2i+1)$ in the order $i=0,1,\ldots,N/2-1$. In pair $i$, row $2i$ is the \emph{donor}, row $2i+1$ is the \emph{receiver}, and row $2i-1$, with indices taken modulo $N$, supplies facilitation.

At the start of each stage, every previously processed odd row is full; every previously processed even row has particles only at odd columns and will not be changed again; and every unprocessed row has its original count and no adjacent vacancies. These properties hold before the first stage. The only horizontal rearrangements outside the current pair will occur in row $N-1$ during stage $0$; they preserve its count and its no-adjacent-vacancies property.

At a fixed stage let $k_{{\rm d}}=k_{2i}(\eta)$ and $k_{{\rm r}}=k_{2i+1}(\eta)$ denote the original donor and receiver counts in this slice, and set $h=N-k_{{\rm r}}$. The assumptions imply  
\begin{equation}
k_{{\rm d}}-\frac{N}{2}\le h<k_{{\rm d}},\qquad h\ge1.\label{eq:plb-transfer-count}
\end{equation}
The first inequality is equivalent to $k_{{\rm d}}+k_{{\rm r}}\le3N/2$; the second follows from $k_{{\rm d}}+k_{{\rm r}}>N$; and $h\ge1$ follows from $k_{{\rm r}}<N$.

Arrange the donor horizontally so that every odd-column site and exactly $b:=k_{{\rm d}}-N/2$ even-column sites are occupied. This target has the required count and no adjacent vacancies, hence is accessible by Lemma \ref{lem:plb-one-dimensional}. Choose a set $Q$ of $h$ occupied donor sites containing all $b$ occupied even-column sites. This is possible by \eqref{eq:plb-transfer-count}. Fix an ordering of $Q$, and perform no further horizontal jumps in the donor during this stage.

Suppose that $j<h$ transfers have been made, and let $x$ be the column of the next selected donor site. Its particle is still present, since the selected sites are distinct and only previously selected donor sites have been changed. The receiver has $k_{{\rm r}}+j$ particles, with $N/2<k_{{\rm r}}+j<N$, and no adjacent vacancies. The latter property is preserved both by the horizontal rearrangements and by filling a vacancy. Lemma \ref{lem:plb-one-dimensional} therefore lets us arrange the receiver so that $(x,2i+1)$ is vacant.

If $i\ge1$, row $2i-1$ is already full, so $(x,2i-1)$ is occupied. For $i=0$, row $N-1$ still has its original count, strictly between $N/2$ and $N$, and has no adjacent vacancies. Rearrange it horizontally to put a particle at $(x,N-1)$, using Lemma \ref{lem:plb-one-dimensional}. Since $N\ge4$, the facilitator row is distinct from both donor and receiver, so this rearrangement preserves the two occupations just specified. In either case the jump  
\[
(x,2i)\to(x,2i+1)
\]
is legal, facilitated by the particle at $(x,2i-1)$. It completes the next transfer and preserves the conditions needed for the following transfer.

After $h$ transfers the receiver is full. Every even-column donor site that was occupied immediately after the donor rearrangement belongs to $Q$ and has now been emptied. Hence all remaining donor particles lie at odd columns. No other row count has changed. Rearrangements in row $N-1$ during stage $0$ preserve its count and its no-adjacent-vacancies property, so that row remains an admissible receiver at the final stage. The stated stage properties therefore hold inductively.

At the end of this slice construction every odd row is full and every even row has particles only at odd columns. In particular, all sites of $\mathbb{H}\cap\Pi_{\bm{w}}$ are vacant. Perform the same finite construction on every remaining slice. The concatenated path uses finitely many legal jumps, never changes an already processed slice, and ends in $\Lambda_{N}$. Closure of $\Lambda_{N}$ under the full process prevents return to the initial state and proves transience.
\end{proof}

\subsection{Graphical Construction and Localization}
For $a,b\in\mathbb{T}_{N}$, let $d_{N}(a,b)=\min_{k\in\mathbb{Z}}|a-b+kN|$, using any integer representatives. For $\bm{x},\bm{z}\in\mathbb{T}^{d}_{N}$, set  
\[
d_{N}(\bm{x},\bm{z})=\max_{j\in\llbracket1,d\rrbracket}d_{N}(x_{j},z_{j}),\qquad B_{r}(\bm{z})=\{\bm{x}:d_{N}(\bm{x},\bm{z})\le r\}.
\]
We use integer radii $r\ge1$ with $4r<N$.

For each $(\bm{x},\bm{v})\in\mathbb{T}^{d}_{N}\times\mathcal{V}$, attach an independent rate-$1$ Poisson clock. At a ring, apply the jump $\bm{x}\to\bm{x}+\bm{v}$ if it is admissible, and otherwise leave the configuration unchanged. This update reads the three sites  
\[
S(\bm{x},\bm{v})=\{\bm{x}-\bm{v},\bm{x},\bm{x}+\bm{v}\}
\]
and can change only the two sites $E(\bm{x},\bm{v})=\{\bm{x},\bm{x}+\bm{v}\}$. There are $2dN^{d}$ clocks and almost surely finitely many rings on every bounded time interval, so these updates construct the process.

For a center $\bm{z}$ and radius $r$, construct $\eta^{\bm{z},r}_{t}$ on the $d$-dimensional box $B_{r}(\bm{z})$ from the same initial occupations in that box. Use the same clocks, but retain a clock only if its entire support $S(\bm{x},\bm{v})$ is contained in $B_{r}(\bm{z})$, and suppress all other clocks. Consequently, $\eta^{\bm{z},r}_{t}(\bm{z})$ depends only on initial occupations in the box and on clocks whose supports are contained there.

\begin{lem}
\label{lem:plb-poisson-tuples}Let $\{\Pi_{a}:a\in\mathcal{L}\}$ be a finite family of independent rate-$1$ Poisson clocks. Fix $t\ge0$, an integer $\ell\ge1$, and a sequence of labels $a_{1},\ldots,a_{\ell}\in\mathcal{L}$, with repetitions allowed. The expected number of tuples $(s_{1},\ldots,s_{\ell})$ such that $t>s_{1}>\cdots>s_{\ell}>0$ and $s_{i}$ is a ring of $\Pi_{a_{i}}$ for each $i$ is $t^{\ell}/\ell!$.
\end{lem}

\begin{proof}
The assertion is immediate when $t=0$, so assume $t>0$. For each label $a$ occurring in the sequence, let $m_{a}$ be its multiplicity and let $C_{a}$ be its number of rings in $(0,t)$. The variables $C_{a}$ are independent Poisson variables of mean $t$. Conditional on $C_{a}=n$, its $n$ ring locations, without their time ordering, can be represented by independent uniform variables on $(0,t)$. These conditional representations are independent for distinct labels.

Write $(n)_{m}=n(n-1)\cdots(n-m+1)$ for $n\ge m$, and $(n)_{m}=0$ for $n<m$. Conditional on all the counts, there are $\prod_{a}(C_{a})_{m_{a}}$ ways to assign distinct rings of each label to its occurrences in the sequence. For every fixed assignment, the $\ell$ assigned times are independent uniform variables, so the probability that they occur in the specified strict order is $1/\ell!$. Hence the conditional expected number of tuples is $(\ell!)^{-1}\prod_{a}(C_{a})_{m_{a}}$. For each $m\ge1$,  
\[
\mathbb{E}[(C_{a})_{m}]=e^{-t}\sum^{\infty}_{n=m}\frac{n!}{(n-m)!}\frac{t^{n}}{n!}=t^{m}e^{-t}\sum^{\infty}_{j=0}\frac{t^{j}}{j!}=t^{m}.
\]
Taking expectations, using independence of the counts and $\sum_{a}m_{a}=\ell$, proves the claim.
\end{proof}

\begin{lem}
\label{lem:plb-localization}
Fix an initial configuration $\eta$, and let $\mathbb P_\eta^N$ also denote the law of the graphical coupling above. For every $t \ge 0$, $\bm z \in \mathbb T_N^d$, and integer $r \ge 1$ with $4r < N$, 
\begin{equation}
\mathbb{P}_\eta^N \left(\eta^{N}_{t}(\bm{z})\ne\eta^{\bm{z},r}_{t}(\bm{z})\right)\le\exp\left((12de)t-\frac{r}{2}\right).\label{eq:plb-localization}
\end{equation}
The bound is uniform over the initial configuration and therefore also holds for any random initial configuration independent of the clocks. 
\end{lem}

\begin{proof}
Clearly, we may assume that $t > 0$.
We work on the probability-one event that the clock rings in $[0,t]$ are finite in number, have distinct times, and do not occur at $0$ or $t$. Suppose that the two processes disagree at $(\bm{z},t)$. At a current site $\bm{u}$ and current time, inspect the latest earlier global clock ring whose endpoint set contains $\bm{u}$. There must be such a ring: otherwise both values would still equal their common initial value.

If its support leaves $B_{r}(\bm{z})$, append any site of that support outside the box and terminate the trace. If the support is contained in the box, both processes apply the same deterministic update map, including the admissibility test. Different output values at $\bm{u}$ imply that at least one of the three input values was different immediately before the ring. Choose such a site and continue the trace from the left limit at that earlier ring time. In particular, the next ring inspected is strictly earlier.

The trace must terminate at a suppressed clock: it cannot use infinitely many rings, and it cannot reach a discrepancy at time zero inside the box. Thus there are sites $\bm{z}=\bm{z}_{0},\bm{z}_{1},\ldots,\bm{z}_{\ell}$ and times $t>s_{1}>\cdots>s_{\ell}>0$ such that $\bm{z}_{\ell}\notin B_{r}(\bm{z})$ and the clock at step $i$ has $\bm{z}_{i-1}$ in its endpoint set and $\bm{z}_{i}$ in its support. Each spatial step has $d_{N}$-length at most $2$, including the last step. Consequently,  
\[
r<d_{N}(\bm{z}_{\ell},\bm{z})\le2\ell,\qquad\ell\ge n_{r}:=\lfloor r/2\rfloor+1.
\]

For a fixed current site, precisely $2d$ clock labels have that site as their source and $2d$ have it as their target. Each of these $4d$ labels permits at most three choices of the next site. Hence the number of possible sequences of labels and sites of length $\ell$, starting at $\bm{z}$, is at most $(12d)^{\ell}$. For a fixed such sequence, Lemma \ref{lem:plb-poisson-tuples} gives expected number $t^{\ell}/\ell!$ of compatible decreasing tuples of ring times, including when a label is repeated. The probability that a tuple exists is at most this expectation. Counting all tuples, without imposing the further conditions of the actual discrepancy trace, only enlarges the event. Therefore a union bound gives  
\begin{align*}
\mathbb{P}\left(\eta^{N}_{t}(\bm{z})\ne\eta^{\bm{z},r}_{t}(\bm{z})\right)\le\sum_{\ell\ge n_{r}}\frac{(12dt)^{\ell}}{\ell!} & \le e^{-n_{r}}\sum_{\ell\ge n_{r}}\frac{((12de)t)^{\ell}}{\ell!}\\
 & \le\exp((12de)t-n_{r})\le\exp\left((12de)t-\frac{r}{2}\right),
\end{align*}
as desired.
\end{proof}

\subsection{Concentration of the Line Densities}

\begin{lem}
\label{lem:plb-row-concentration}Let $\rho\in(0,1)$, $t\ge0$, and let $r\ge1$ be an integer with $4r<N$. Set  
\[
Q=4r+1,\qquad\delta=\exp\left((12de)t-\frac{r}{2}\right).
\]
For every $\gamma>0$ with $\delta\le\gamma/2$ and every line $L_{\bm{y}}$, $\bm{y}\in\mathbb{T}^{d-1}_{N}$,  
\begin{equation}
\mathbb{P}^{N}_{\nu_{\rho}}\left(\left|\frac{k_{\bm{y}}(\eta^{N}_{t})}{N}-\rho\right|\ge\gamma\right)\le N\delta+2\exp\left(-\frac{\gamma^{2}N}{2Q}\right).\label{eq:plb-row-concentration}
\end{equation}
\end{lem}

\begin{proof}
Translating the initial configuration and every clock label by a fixed torus vector translates the graphical trajectory by that vector. The product initial law and the joint clock law are invariant under this relabeling, so the time-$t$ law is translation invariant. By conservation of the total particle number and translation invariance,  
\begin{equation}
\mathbb{E}^{N}_{\nu_{\rho}}\left[\eta^{N}_{t}(\bm{x})\right]=\frac{1}{N^{d}}\mathbb{E}^{N}_{\nu_{\rho}}\left[|\eta^{N}_{t}|\right]=\rho\qquad\text{for every}\quad\bm{x}\in\mathbb{T}^{d}_{N}.\label{eq:plb-marginal}
\end{equation}
This identity does not require the time-$t$ law to be a product measure.

Fix the line $L_{\bm{y}}$. For $x\in\mathbb{T}_{N}$, put $\bm{z}_{x}=(x,\bm{y})$ and define  
\[
Z_{x}=\eta^{N}_{t}(\bm{z}_{x}),\qquad\widehat{Z}_{x}=\eta^{\bm{z}_{x},r}_{t}(\bm{z}_{x}),\qquad\mu_{x}=\mathbb{E}^{N}_{\nu_{\rho}}\left[\widehat{Z}_{x}\right].
\]
Every local occupation is $\{0,1\}$-valued at every time, so $\widehat{Z}_{x}$ has the Bernoulli distribution with parameter $\mu_{x}=\mathbb{P}^{N}_{\nu_{\rho}}(\widehat{Z}_{x}=1)$. Its parameter need not equal $\rho$. Lemma \ref{lem:plb-localization} and \eqref{eq:plb-marginal} imply  
\begin{equation}
\mathbb{P}^{N}_{\nu_{\rho}}\left(Z_{x}\ne\widehat{Z}_{x}\right)\le\delta,\qquad|\mu_{x}-\rho|\le\mathbb{E}^{N}_{\nu_{\rho}}\left|\widehat{Z}_{x}-Z_{x}\right|=\mathbb{P}^{N}_{\nu_{\rho}}\left(\widehat{Z}_{x}\ne Z_{x}\right)\le\delta.\label{eq:plb-local-means}
\end{equation}

Join distinct vertices $x,x'\in\mathbb{T}_{N}$ when $d_{N}(x,x')\le2r$. Since $4r<N$, each vertex has the $4r$ distinct neighbors $x\pm1,\ldots,x\pm2r$. Greedy coloring gives $Q=4r+1$ color classes $I_{1},\ldots,I_{Q}$, some possibly empty, such that $d_{N}(x,x')>2r$ for distinct $x,x'$ in the same class. If $B_{r}(\bm{z}_{x})$ and $B_{r}(\bm{z}_{x'})$ intersected, the first coordinate of an intersection point and the triangle inequality would give $d_{N}(x,x')\le2r$. Thus the $d$-dimensional boxes in each color class are pairwise disjoint. The number of colors remains $4r+1$: only centers on a single coordinate line are being colored, regardless of the box dimension.

For $x\in I_{j}$, $\widehat{Z}_{x}$ uses only initial occupations in its box and clocks whose entire supports are in that box. The initial-site collections for distinct boxes in the class are disjoint, as are their clock-label collections. Product initial data, independent clocks, and independence of the initial configuration from the clocks therefore give joint independence of $\{\widehat{Z}_{x}:x\in I_{j}\}$. No independence between different classes is used.

For any Bernoulli variable $Y$, one has 
\begin{equation}
\log\mathbb{E}\exp\{\theta(Y-\mathbb{E}Y)\}\le\frac{\theta^{2}}{8}\qquad\text{for every}\quad\theta\in\mathbb{R}.\label{eq:plb-bernoulli-mgf}
\end{equation}
Indeed, if $p=\mathbb{E}Y$, the left side is $h(\theta)=\log(1-p+pe^{\theta})-p\theta$. It satisfies $h(0)=h'(0)=0$ and  
\[
h''(\theta)=p_{\theta}(1-p_{\theta})\le\frac{1}{4},\qquad p_{\theta}=\frac{pe^{\theta}}{1-p+pe^{\theta}}.
\]
The integral Taylor formula $h(\theta)=\theta^{2}\int^{1}_{0}(1-s)h''(s\theta)\,{\rm d}s$ proves \eqref{eq:plb-bernoulli-mgf} for both signs of $\theta$, including the degenerate cases $p=0,1$.

Set  
\[
S_{j}=\sum_{x\in I_{j}}(\widehat{Z}_{x}-\mu_{x}),\qquad S=\sum^{Q}_{j=1}S_{j}.
\]
For every real $\theta$, generalized Hölder's inequality with all exponents equal to $Q$, followed by the independence within each color class and \eqref{eq:plb-bernoulli-mgf}, gives  
\begin{align*}
\mathbb{E}^{N}_{\nu_{\rho}}[e^{\theta S}]\le\prod^{Q}_{j=1}\left(\mathbb{E}^{N}_{\nu_{\rho}}[e^{Q\theta S_{j}}]\right)^{1/Q} & =\prod^{Q}_{j=1}\left[\prod_{x\in I_{j}}\mathbb{E}^{N}_{\nu_{\rho}}e^{Q\theta(\widehat{Z}_{x}-\mu_{x})}\right]^{1/Q}\\
 & \le\prod^{Q}_{j=1}\exp\left(\frac{Q\theta^{2}|I_{j}|}{8}\right)=\exp\left(\frac{QN\theta^{2}}{8}\right).
\end{align*}
For $a>0$ and $\theta>0$, exponential Markov inequality yields  
\[
\mathbb{P}^{N}_{\nu_{\rho}}\left(S\ge aN\right)\le\exp\left(-\theta aN+\frac{QN\theta^{2}}{8}\right).
\]
The choice $\theta=4a/Q$ makes this $\exp(-2a^{2}N/Q)$. The same moment bound at the negative parameter $-\theta$ gives  
\[
\mathbb{P}^{N}_{\nu_{\rho}}\left(S\le-aN\right)\le e^{-\theta aN}\mathbb{E}^{N}_{\nu_{\rho}}e^{-\theta S}\le\exp\left(-\theta aN+\frac{QN\theta^{2}}{8}\right).
\]
Using the same positive $\theta=4a/Q$ and adding the two tails gives 
\begin{equation}
\mathbb{P}^{N}_{\nu_{\rho}}\left(|S|\ge aN\right)\le2\exp\left(-\frac{2a^{2}N}{Q}\right).\label{eq:plb-colored-tail}
\end{equation}

Let $G=\{Z_{x}=\widehat{Z}_{x}\quad\text{for every}\enspace x\}$. A union bound gives $\mathbb{P}^{N}_{\nu_{\rho}}(G^{c})\le N\delta$. On $G$,  
\[
\left|\frac{k_{\bm{y}}(\eta^{N}_{t})}{N}-\rho\right|\le\frac{|S|}{N}+\left|\frac{1}{N}\sum_{x}(\mu_{x}-\rho)\right|\le\frac{|S|}{N}+\delta.
\]
Since $\delta\le\gamma/2$, a deviation by at least $\gamma$ on $G$ implies $|S|\ge\gamma N/2$. Apply \eqref{eq:plb-colored-tail} with $a=\gamma/2$ and add $\mathbb{P}^{N}_{\nu_{\rho}}(G^{c})$ to obtain \eqref{eq:plb-row-concentration}.
\end{proof}

\subsection{Proof of Theorem \ref{thm:plb-main}}

\begin{proof}
[Proof of Theorem \ref{thm:plb-main}] Fix $d\ge2$, even $N$, and $\rho\in(1/2,3/4)$, and write $\gamma=\gamma_{\rho}$ and $c=c_{\rho,d}$ as in \eqref{eq:plb-constants}. For every deterministic $t\ge0$, almost surely,  
\begin{equation}
\left\{ \tau^{N}_{{\rm tr}}\le t\right\} =\left\{ \eta^{N}_{t}\in\Sigma^{{\rm rec}}_{N}\right\} .\label{eq:plb-hitting-identity}
\end{equation}
If $|k_{\bm{y}}(\eta^{N}_{t})/N-\rho|<\gamma$ on every line, all these line densities lie strictly between $1/2$ and $3/4$. Lemma \ref{lem:plb-transient} therefore shows that $\eta^{N}_{t}$ is transient. Consequently,  
\begin{equation}
\left\{ \tau^{N}_{{\rm tr}}\le t\right\} \subseteq\bigcup_{\bm{y}\in\mathbb{T}^{d-1}_{N}}\left\{ \left|\frac{k_{\bm{y}}(\eta^{N}_{t})}{N}-\rho\right|\ge\gamma\right\} .\label{eq:plb-transience-reduction}
\end{equation}
There are exactly $N^{d-1}$ lines in this union. For any admissible radius $r$ with $\delta=\exp((12de)t-r/2)\le\gamma/2$, Lemma \ref{lem:plb-row-concentration} thus gives  
\[
\mathbb{P}^{N}_{\nu_{\rho}}\left(\tau^{N}_{{\rm tr}}\le t\right)\le N^{d}\delta+2N^{d-1}\exp\left(-\frac{\gamma^{2}N}{2(4r+1)}\right).
\]
Choose 
\begin{equation}
r_{N}(t)=\left\lceil (24de)t+2(d+4)\log N\right\rceil .\label{eq:plb-radius}
\end{equation}
Whenever $4r_{N}(t)<N$, its localization error satisfies  
\[
\delta_{N}(t):=\exp\left((12de)t-\frac{r_{N}(t)}{2}\right)\le N^{-(d+4)}.
\]
For $N$ large enough that $N^{-(d+4)}\le\gamma/2$, this yields  
\begin{equation}
\mathbb{P}^{N}_{\nu_{\rho}}\left(\tau^{N}_{{\rm tr}}\le t\right)\le N^{-4}+2N^{d-1}\exp\left(-\frac{\gamma^{2}N}{2(4r_{N}(t)+1)}\right).\label{eq:plb-quantitative}
\end{equation}
Set $t=t_{N}=cN/\log N$. Since $d$ and $c>0$ are fixed, $r_{N}(t_{N})/N\to0$, so $4r_{N}(t_{N})<N$ for all sufficiently large $N$. Moreover,  
\[
4r_{N}(t_{N})+1\le(96de)c\frac{N}{\log N}+8(d+4)\log N+5\le(192de)c\frac{N}{\log N}
\]
eventually, because $(\log N)^{2}/N\to0$. The definition of $c$ gives  
\[
\frac{\gamma^{2}N}{2(4r_{N}(t_{N})+1)}\ge\frac{\gamma^{2}}{(384de)c}\log N=(d+2)\log N.
\]
Thus the second term in \eqref{eq:plb-quantitative} is at most $2N^{d-1}N^{-(d+2)}=2N^{-3}$, proving \eqref{eq:plb-main-bound}.
\end{proof}

\begin{acknowledgement}
The authors declare that generative AI (GPT-6 Astra) tools were used solely for language polishing and editorial revision of the manuscript.
S. Kim has been supported by the Basic Science Research Program through the National Research Foundation of Korea funded by the Ministry of Science and ICT (RS-2025-00518980, RS-2026-25518141), the Yonsei University Research Fund of 2026 (2026-22-0181), and the POSCO Science Fellowship of POSCO TJ Park Foundation. S. Lee and I. Seo have been supported by National Research Foundation of Korea funded by the Ministry of Science and ICT (RS-2025-23525546, RS-2023-NR076621, and RS-2026-25518141).
\end{acknowledgement}


\begin{thebibliography}{99}
\bibitem{BG90}C. Bezuidenhout and G. Grimmett. \emph{The Critical Contact Process Dies Out}. Ann. Probab., 18:1462--1482, 1990.

\bibitem{BEKL26}O. Blondel, C. Erignoux, S. Kim and S. Lee. \emph{Sharp freezing time estimates for the subcritical Facilitated Exclusion Process}. arXiv:2606.15233, 2026.

\bibitem{BESS20}O. Blondel, C. Erignoux, M. Sasada and M. Simon. \emph{Hydrodynamic limit for a facilitated exclusion process}. Ann. Inst. Henri Poincar\'e Probab. Stat., 56:667--714, 2020.

\bibitem{BES21}O. Blondel, C. Erignoux and M. Simon. \emph{Stefan problem for a nonergodic facilitated exclusion process}. Probab. Math. Phys., 2:127--178, 2021.

\bibitem{Bol06}B. Bollob\'as. \emph{The Art of Mathematics: Coffee Time in Memphis}. Cambridge University Press, 2006.

\bibitem{DCES26}H. Da Cunha, C. Erignoux and M. Simon. \emph{Hydrodynamic limit for an open facilitated exclusion process with slow and fast boundaries}. Comm. Math. Phys., 407:50, 2026.


\bibitem{ERSS24}C. Erignoux, A. Roget, A. Shapira and M. Simon. \emph{Hydrodynamic behavior near dynamical criticality of a facilitated conservative lattice gas}. Phys. Rev. E, 110:L032101, 2024.

\bibitem{Har74}T. E. Harris. \emph{Contact Interactions on a Lattice}. Ann. Probab., 2:969--988, 1974.

\bibitem{Rol20}L. T. Rolla. \emph{Activated Random Walks on $\mathbb Z^d$}. Probab. Surv., 17:478--544, 2020.

\bibitem{RS12}L. T. Rolla and V. Sidoravicius. \emph{Absorbing-state phase transition for driven-dissipative stochastic dynamics on $\mathbb Z$}. Invent. Math., 188:127--150, 2012.

\bibitem{ST17}V. Sidoravicius and A. Teixeira. \emph{Absorbing-state transition for Stochastic Sandpiles and Activated Random Walks}. Electron. J. Probab., 22, no.~33, 1--35, 2017.

\bibitem{VDMZ00}A. Vespignani, R. Dickman, M. A. Mu\~noz and S. Zapperi. \emph{Absorbing-state phase transitions in fixed-energy sandpiles}. Phys. Rev. E, 62:4564--4582, 2000.



\end{thebibliography}
\end{document}